\documentclass[12pt,a4paper,leqno]{article}
\usepackage{amsmath}
\usepackage{amssymb}
\usepackage{amsfonts}
 \usepackage{amsthm}
\usepackage[pdftex]{graphicx}
\usepackage[english]{babel}
\usepackage{mathrsfs}
\usepackage{xypic}
\usepackage{overpic}
\usepackage{enumitem}
\usepackage[applemac]{inputenc}
\usepackage{float}
 \usepackage{xy}
 \usepackage{xcolor}
\usepackage{pdfsync}%
\numberwithin{equation}{section}
\usepackage[useregional]{datetime2}
\renewcommand{\phi}{\varphi}
\renewcommand{\tilde}{\widetilde}
\usepackage{hyperref}

\DeclareMathOperator{\vol}{vol}
\DeclareMathOperator{\Id}{Id}
\DeclareMathOperator{\Cal}{Cal}
\DeclareMathOperator{\supp}{supp}
\def\GFQI {Generating Function Quadratic at Infinity }
\def \GFQIs {Generating Functions Quadratic at Infinity }

\def \ZZ {{\mathbb Z}}

\def \Ham {\rm Ham}
\def \DHam {\rm {DHam}}
\def \cF {\mathcal F^\bullet}
\def \F {\mathcal F}
\def \cG {\mathcal G^\bullet}
\def \cI {\mathcal I^\bullet}

\def \meas {\mathit m}
\renewcommand{\theenumi}{(\alph{enumi})}

\begin{document}
\newtheorem{theorem}{Theorem}[section]
\newtheorem{lemma}[theorem]{Lemma}
\newtheorem{proposition}[theorem]{Proposition}
\newtheorem{Cor}[theorem]{Corollary}
\newtheorem{definition}[theorem]{Definition}
\def \cal {\mathcal}
\def \dispdot {}

\theoremstyle{definition}
\newtheorem{example}[theorem]{Example}
\newtheorem{exer}[section]{Exercise}
\newtheorem{conj}{Conjecture}

\theoremstyle{remark}
\newtheorem{remark}[theorem]{Remark}
\newtheorem{remarks}[theorem]{Remarks}
\newtheorem{question}{Question}

\font\smallcaps=cmcsc10
\font\nome=cmr8

\title{Non-convex Mather's theory and the Conley conjecture on the cotangent bundle of the torus}
\author{ { \sc Claude Viterbo}
\\
D\'epartement  de Math\'ematique\\ Universit\'e de Paris-Sud, Paris-Saclay, Orsay, France
\\ {\tt claude.viterbo@universite-paris-saclay.fr}
\thanks{Part of this work was done while the author was at Centre de Math\'ematiques  Laurent Schwartz, UMR 7640 du CNRS,
\'Ecole Polytechnique - 91128 Palaiseau, France and  DMA,  UMR  8553 du CNRS
 \'Ecole Normale Sup\'erieure-PSL University, 45 Rue d'Ulm, 75230 Paris Cedex 05,  France. Supported also by Agence Nationale de la Recherche projects Symplexe,  WKBHJ and MICROLOCAL. }
}
\date{ \today \; at \DTMcurrenttime}
\maketitle

\begin{abstract}The aim of this paper is to use the methods and results of symplectic homogenization (see \cite{STH}) to prove existence of periodic orbits and invariant measures with rotation vector belonging to the set of differentials of the Homogenized Hamiltonian.   We also prove the Conley conjecture on the cotangent bundle of the torus. Both proofs rely on Symplectic Homogenization and a refinement of it. \end{abstract}
 \tableofcontents

 \section{Introduction} 
The symplectic theory of Homogenization, set up in \cite{STH},  associates to each Hamiltonian $H(t,q,p)$ on $T^*T^n$ a homogenized Hamiltonian, $\overline H (p)$, such that   $H_k(t,q,p)=H(kt,kq,p)$ $\gamma$-converges to $\overline H (p)$, where the metric $\gamma$ , a symplectic metric defined in \cite{Viterbo-STAGGF}, will be defined later \footnote{see \cite{Viterbo-STAGGF},  and the related  Hofer metric in \cite{Hofer}. See also \cite{Humiliere}  for the study of this metric and its completion mentioned further.}.  The goal of this paper is to draw some dynamical consequences of the homogenization theorem, to prove existence of certain trajectories of the flow $\varphi^t$ and then of invariant measures.
We also apply this methods to the proof of the Conley conjecture\footnote{according to \cite{Salamon-Zehnder} this was stated-at least for closed manifolds- by C. Conley in \cite{Conley-2}.} on $T^*T^n$ , claiming the existence of infinitely many distinct and non-trivial periodic points for the time-one map of a compact supported Hamiltonian map in $T^*T^n$. 

Symplectic Homogenization may be summarized as the following heuristic statement

\bigskip 
{\bf Symplectic Homogenization Principle:} {\it The value of any variational problem associated to $H_k$  will converge to the value of the same variational problem associated to $\overline H$. }
\bigskip 

While the above sentence is vague and does not claim to be a mathematical statement, we hope it carries sufficient meaning for the reader to help him understand  the substance of the method used in the present paper.

If we denote by  $\varphi^t$ the flow associated to the Hamiltonian $H$, and by $\overline\varphi ^t$ the "flow" of $\overline H(p)$ -- defined in the Humili\`ere completion $\widehat \DHam(T^*T^n)$ of the group of Hamiltonian diffeomorphisms of $T^*T^n$ for the metric $\gamma$, then $\overline \varphi^t$  is the $\gamma$-limit of $\rho_{k}^{-1}\varphi^{kt}\rho_{k}$ where $\rho_{k}(q,p)=(kq,p)$ and $\varphi^t$ is the flow of the Hamiltonian vector field associated to $H$.
Note that $\rho_k^{-1}$ is not well-defined, but $\varphi_k=\rho_{k}^{-1}\varphi^{kt}\rho_{k}$ is well defined if $\varphi$ is Hamiltonianly isotopic to the identity, as the unique solution of $\rho_k\varphi_k=\varphi^k\rho_k$ obtained by continuation starting from $\varphi=\varphi_k=\Id$.

\section{Notations} Let $H(t,q,p)$ be a Hamiltonian on $T^*T^n$. We denote by $\varphi$ the time-one flow $\varphi^1$ of $X_H$, where $X_H$ is defined by $\omega(X_H,  )= -dH$, the Hamiltonian vector field associated to $H$. 

We denote  by $\Phi^t$ the lift of $\varphi^t$ to the universal cover $ {\mathbb R}^{2n}$ of $T^*T^n$ (or $\Phi_H^t, \varphi_H^t$ if there is any ambiguity). 

The action of a trajectory $\gamma (t)=(q(t), p(t))=\varphi^t(q(0),p(0))$ defined on $[0,1]$ is $$A(\gamma)=\int_0^1 [p(t)\dot q(t) -H(t,q(t),p(t)) ]dt \dispdot $$ The average action for a solution defined on $[0,T]$ is $$A_{T}(\gamma)=\frac{1}{T}\int_0^T [p(t)\dot q(t) -H(t,q(t),p(t)) ]dt \dispdot $$
We denote  by $\lambda$ the Liouville form $pdq$ on $T^*T^n$ and write $\gamma^*\lambda=p\dot q dt$  for $\gamma=(q,p)$, so that the action is also given by $$A(\gamma)=\int_0^1[\gamma^*\lambda-H(t,\gamma(t))]dt$$
 Finally if $f$ is a function on a manifold $M$ we denote by $$f^c=\{x\in M \mid f(x)<c\}$$ 
 For the definition of the strong differential $d_sf(x)$ we refer to Appendix \ref{Appendix-crit}, Definition \ref{Def-9.14}. Then $D_sf(x)$ is the set of limits of $d_sf(x_k)$ for all sequences $x_k$ converging to $x$ and  the Clarke differential $\partial_Cf(x)$ will be found in Definition \ref{Def-12.1}.

\section{Statement of the main results}
Our goal is to prove the following theorems. The rotation vector of an invariant probability measure on $S^1\times T^*M$  is defined just after the statement of the theorem. 

\begin{theorem} \label{main-theorem}
Let $H(t,q,p)$ be a compact supported Hamiltonian in $S^{1}\times T^{*}T^{n}$ 
and denote by $\overline H (p)$ its homogenization defined in \cite{STH}. 
Let   $\alpha \in d_s \overline H (p)$. 
Then there exists a sequence   $(\alpha_k)_{k\geq 1}$ such that   $\lim_{k}\alpha_{k}= \alpha$ and for infinitely many $k$ a solution $(q_k,p_k, p'_k)$  of $\Phi^{k}(q_{k},p_{k})=(q_{k}+ k \alpha_{k} ,p'_{k})$. Moroever as $k$ goes to infinity the average action 
$$A_k = \frac{1}{k} \int_{0}^k [\gamma_{k}^*\lambda -H(t,\gamma_{k}(t)) ]dt$$ of the orbit $\gamma_{k} (t)= \varphi^{t}(q_{k},p_{k})$  converges
  to $$\lim_{k}A_{k}=\langle p, \alpha \rangle -\overline H(p)$$ Therefore for each $\alpha \in \partial_C {\overline H}(p)$ there exists an invariant 
 probability measure $\meas_{\alpha}$,  having compact support in $T^*T^n$, with rotation vector $\alpha$ and average action  
  $${\mathcal A}(\meas_{\alpha})\overset{def}=\int_{S^1\times T^*T^n} [p\frac{\partial H}{\partial p}(t,q,p)-H(t,q,p) ]d\meas_{\alpha}= \langle p, \alpha\rangle-\overline H(p).$$
\end{theorem} 

\begin{remarks}
\begin{enumerate} 
\item  All our invariant measures will be probability measures.  
 First of all,  remember that a measure $m$  is invariant if, setting $\psi^t(s,z)=(t+s, \varphi^t(z))$, we have $(\psi^1)_{*}(\meas)=\meas$. For $\alpha \in d_s\overline H(p)$, once we found the sequence $(q_k,p_k)$, the measure $\meas_\alpha$ is obtained with the usual Krylov-Bogolioubov method as the limit of the measures uniformly distributed on $\{\psi^{kt}(0,q_k,p_k) \mid t\in[0,1]\}$. The extension to $\alpha \in \partial_C\overline H(p)$ is obtained by taking convex hulls of Krylov-Bogolioubov measures. 

\item As detailed in \cite{Mather}, page 176, the rotation vector (see \cite{Schwartzman}) of the  measure $\meas$ is the element of $H_{1}(T^n,  {\mathbb R} )$  given by  $$r(\meas)=\int_{S^1\times T^*T^n}\frac{\partial H}{\partial p}(q,p) d\meas$$ or equivalently, in its dual form, as the map 
 \begin{align*} r(\meas):&  H^1(T^n, {\mathbb R} ) \longrightarrow {\mathbb R} \\ & \tau    \longrightarrow \int_{S^1\times T^*T^n} i_{X_H}\tau d\meas =
 \int_{S^1\times T^*T^n} \langle \tau(q), \frac{\partial H}{\partial p} \rangle d\meas  =\\ & \qquad \int_{S^1\times T^*T^n}\sum_{j=1}^n  \tau_j \frac{\partial H}{\partial p_j}(t,q,p) d\meas  \end{align*} 
Here we identify the closed $1$-form $\tau$ with its cohomology class and the form on $T^n$  to its pull-back on $T^*T^n$ ! It is well-known that the right-hand side only depends on this cohomology class and not on the specific choice of a representative. 

\item We will  define $d_s\overline H(p)$ and $\partial_C \overline H (p)$ in Appendix \ref{Appendix-crit}. As we pointed out in \cite{STH}, we cannot hope in general\footnote{Of course if $H$ is integrable (i.e. only depends on $p$) then $\overline H=H$ and we can get better regularity. An other example is the pendulum, for which $\overline H$ is smooth except for $p=A$ where $A$ is the area enclosed by the two separatrices.} anything better than $\overline H$ to be in $C^{0,1}$. It is thus important to figure out the set $\partial_C \overline H(p)$ when $\overline H$ is not differentiable at $p$. 
\end{enumerate} 
\end{remarks} 

Of course the Liouville measure $\omega^n$ is invariant. While it is not a probability measure,  since $\varphi$ is compact supported, we may truncate $\omega^n$ to $\chi ( q,p)\cdot \omega^n$, where $\chi\equiv \frac{1}{\vol(\supp(\varphi))}$ on the support\footnote{Her by abuse of language we mean by support of $\varphi$ the support of $H$. The constant is chosen to ensure that $\meas$ has integral $1$.} of $\varphi$ and $\chi\equiv 0$ elsewhere.  We call this the {\bf Liouville invariant measure}. 

However, 
\begin{lemma} \label{Lemma-3.4}
For $\meas$ the Liouville invariant measure  we have $r(\meas)=0$ and  $$\mathcal A (\meas)=-(n+1) \frac{ \Cal(\varphi)}{\vol(\supp(\varphi))}$$ where $\Cal(\varphi)$ is the Calabi invariant defined by $$\Cal (\varphi)=\int_{S^1\times T^n} H(t,q,p) \omega^n \wedge dt$$ and $\vol(U)=\int_U \omega^n$. 
\end{lemma} 

Note that  if $\alpha=0$, ${\mathcal A}(\meas_{0})=-\overline H(p_{0})$, where $p_{0}$ is a critical point\footnote{Since $\overline H$ is only continuous, we mean $0\in \partial_C\overline H(p_0)$. See the definition of $\partial_C\overline H$ in Appendix (section \ref{sec:10}).} of $\overline H$. So if for some $\alpha$, $m_\alpha=\chi (q,p) \omega^n$ is  the standard measure,
 it must correspond to  $\alpha=0$ and we must have the equality $$\overline H(p_{0})= \frac{-(n+1)\Cal (\varphi)}{\vol (\supp(\varphi))} $$ 

\begin{remark} 
Note that if the measure obtained for $\alpha=0$ is the trivial measure, we may draw some precise conclusions. 

Indeed since our invariant measure is obtained by the Krylov-Bogolyoubov method, by taking limits of the measure carried by orbits, 
the support of $\chi$ must be contained in the support of $H$, since outside the support, the Krylov-Boglyoubov method yields the Dirac mass at a point. 

 In other words 
 the flow must then have the following property : there exists a sequence $z_k$ of points such that the normalized measure on $\{\Phi^{kt}(z_k) \mid t\in [0,1]\}$ converges to the Liouville measure on the support of $\varphi$. Since $z_k$ belongs to the support of  $H$, taking subsequences we may assume $z_k$ converges to some point $z$. Then any neighborhood $U$ of $z$ contains a point whose orbit comes  arbitrarily close to any other point in the support of $\chi$. In other words, $\Phi^{[0,+\infty[}(U)$ is dense in the support of $H$. Applying Baire's theorem to the sequence of $\Phi^{[0,+\infty[}(B(z, \frac{1}{n} ))$ we get that $\Phi^{[0,+\infty[}(z)$ is dense, hence the  flow is topologically transitive on its support\footnote{The equivalence of topological transitivity and the existence of a dense orbit in this setting is established in \cite{Silverman-DO-TT}.}.
\end{remark} 

Remember also that $\overline H$ coincides with Mather's $\alpha$ function when $H$ is strongly convex in $p$, that is $D_{pp}^2H(q,p)\geq \varepsilon \Id$ for some $\varepsilon >0$ (see \cite{STH}, section 12.2), but in this case the set of values of $\partial_C \overline H (p)$ as $p$ describes $ {\mathbb R}^n$  is the whole of $ {\mathbb R} ^n$, so we get any rotation vector, as expected from standard Aubry-Mather theory (see \cite{Mather}).  
This may be generalized to the following : we say that $H$ is superlinear if $$\lim_{ \vert p \vert \to \infty}  \frac{ H(q,p)}{\vert p \vert}=+\infty$$
Note that in \cite[section 11.1]{STH}, we showed that $\overline H$ is well defined beyond the compact supported case, in particular in the coercive case. 
\begin{Cor} 
Let $H(q,p)$ be a superlinear Hamiltonian on $T^*T^n$. Then $\overline H$ is superlinear, so that for any $\alpha \in {\mathbb R} ^n$, we may find an invariant measure  for the flow, with rotation vector $\alpha$. 
\end{Cor} 
\begin{proof} Using Proposition 11.9 and 11.10 in \cite{STH}, we may extend the definition of the Homogenization to the coercive case. 
Superlinear means that for any positive $C$ there exists $A$ such that $H(q,p)\geq C \vert p \vert -A$. But this implies $\overline H(p)\geq C\vert p \vert -A$, hence $\overline H$ is superlinear. Now if $\overline H$ is superlinear, $p \mapsto \overline H(p)- \langle \alpha , p\rangle$ is bounded from below (use the inequality $\overline H(p)\geq C\vert p \vert -A$ with $C > \vert \alpha\vert$). It is easy to see that the minimum is achieved, since the set of $p$ such that $\overline H(p)- \langle \alpha , p\rangle \leq r$ is compact. At the minimum, $p_\alpha$, we have $\alpha \in \partial_C\overline H(p_\alpha)$. The rest of the proof is as in the compact-supported case by using a truncation as in \cite[section 11]{STH}. 
\end{proof} 
The main idea of the proof of Theorem 3.1 is to formulate the existence of intersection points in $\Phi ^k(\{q_{0}\}\times {\mathbb R} ^{n})\cap (\{q_{0}+k \alpha \}\times {\mathbb R} ^{n})$ as a  variational problem and apply   our heuristic principle -- i.e. that  a variational problem involving $H_{k}$ must converge to the variational problem involving $\overline H$. 

Another consequence of our methods will be the following result, implying the extension of the Conley conjecture to $T^*T^n$:
 \begin{theorem} \label{Thm-weak-Conley}
 Let us assume $H(t,q,p)$ is a compact supported Hamiltonian on $T^*T^n$.
 
 \begin{enumerate} [label=\theenumi]
 \item  Let $ \frac{u}{v} \in d_s\overline H (p)$ where $u\in {\mathbb Z}^n, v \in \mathbb N^*$ . Then there exists  a $v$-periodic orbit of $\varphi_H$ with rotation vector   $ \frac{u}{v}$, i.e.  a point $z$ such that $\Phi_H^{kv}(z)=z+ku$. 
  \item Assume $\overline H \not\equiv 0$. Then there exist infinitely many geometrically distinct non-contractible periodic orbits for $\varphi^1$. \item Assume $\overline H \equiv 0$ and $\varphi\neq \Id$. Then there exists infinitely many geometrically distinct contractible periodic orbits for $\varphi^1$ contained in the interior of the support of $H$, and moreover if the number of fixed points of $\varphi^1$ contained in the interior of the support of $H$ is finite,  for all positive $ \varepsilon $, we have for $N$ large enough $$\#\{x  \mid \varphi(x)\neq x,\; \text{and}\;\exists k \in [2,N] \mid  \varphi^k(x)=x\} \geq \frac{(1- \varepsilon )N^2}2{\log (N)} $$
 If we only count orbits (i.e. identify $x$ and $\varphi(x)$), then 
 $$\#\{ \text{orbits of}\; \varphi \mid card (\{\varphi^j(x) \mid j \in {\mathbb N} \})  \in [2,N] \} \geq\frac{(1- \varepsilon )N}{\log (N)}$$
 \end{enumerate} 

 \end{theorem} 
 
 Note that both cases: non-existence of contractible non-trivial periodic orbits\footnote{Think of the reparametrized geodesic flow for the flat metric given by $H(q,p)=h( \vert p \vert )$ where $h$ is constant near $0$, then strictly decreasing and then vanishing outside $[0,1]$.  Then its flow has only non-contractible periodic orbits outside its support.} and non-existence of non-contractible ones (for example if $\supp (H)$ is contractible) are possible.  Our result could be considered a generalization of the main result in \cite{Biran-Polterovich-Salamon}, where the first statement is proved under the assumption that $H$ is bounded from below on a certain Lagrangian submanifold. But this assumption implies, according to \cite{STH} that $\overline H$ is nonzero. 
 
The  Conley's conjecture proved by N. Hingston on $T^{2n}$ and on more general compact manifolds by V. Ginzburg (see \cite{Hingston}, \cite{Ginzburg}, \cite{GiG, Hei}) yields existence of infinitely many contractible periodic orbits for a Hamiltonian on $(M,\omega)$. For Lagrangian systems in the  cotangent bundle of a compact manifold, the analogous statement  was proved by Y. Long and G.Lu (\cite{Long-Lu}) for the torus and by G.Lu (\cite{Lu}) in the general case (see also \cite{Abbondandolo-Figalli}, and \cite{Mazzucchelli}) and by Hein (\cite{Hei2}) for asymptotically quadratic Hamiltonians.  
 
 \begin{remark} 
 If $\varphi^t(x)$ is an orbit of period $k$, we denote by $\nu(x,\varphi)$ the vector obtained by considering the  $q$ component of $\frac{1}{k}(\Phi^k(q,p)-(q,p)) \in \frac{1}{k} {\mathbb Z}^{2n}$. Then if $\overline H \not\equiv  0$, we shall prove (see Appendix \ref{Appendix-crit}, Section \ref{Section-9} as a consequence of Lemma \ref{fundlemma}) that the set of limit points of $\nu(x,\varphi)$ as $x$ belongs to the set of $k$-periodic orbits is a subset of ${\mathbb R} ^n$ of non-empty interior. 
 \end{remark} 
 
 A final comment is in order. In the convex case, Aubry-Mather theory makes two claims:
 
 \begin{enumerate}[label=\theenumi]
  \item existence of invariant measures with given rotation vectors
 \item  the support of this invariant measure is a Lipschitz graph over the base of the cotangent bundle
 \end{enumerate}  
 While we believe that the present work gives the right extension of the first statement to non-convex situations (for the moment only in $T^*T^n$ and not in a general cotangent bundle, see however \cite{Vichery3} for more general cotangent bundles and \cite{Bi} for general symplectic manifolds), we make no claim coming  close to the second statement. Of course if we start with a convex Hamiltonian, and compose it with a symplectic map, we get an invariant measure for the new Hamiltonian that is the image by the symplectic map of the invariant measure for the convex Hamiltonian, and this image has no reason to be a graph. 
 
  However other statements could make sense. One plausible conjecture is to look at the action of $\varphi^t$ over $\widehat {\mathcal L}$, the Humili\`ere completion of the set $\mathcal L$ of Lagrangians submanifolds for the $\gamma$-metric. Note that objects in $\widehat {\mathcal L}$ are not submanifolds, they are just objects in an abstract completion\footnote{However on can define their   support as a subset of the ambient symplectic manifold (see \cite{Humiliere}, section 4.4, page 390)}. Indeed, the group of Hamiltonian diffeomorphisms acts on this set (since a Hamiltonian diffeomorphism acts as an isometry for  $\gamma$, over $\mathcal L$, hence acts over its completion). If there is an element $L$, in $\widehat {\mathcal L}$ fixed by $\varphi^t$, then $L$ is not a Lagrangian, but the corresponding graph selector, $u_{L}(x)$ is again well-defined as a Lipschitz function, hence differentiable a.e. The set of points $(x,du_{L}(x))$ where $u_{L}$ is differentiable could then contain a set  invariant by the flow $\varphi^t$. 
 Note that the approach in this paper is very far from this conjecture, since we obtain the invariant measure as a limit of measures supported on trajectories, and there is no obvious way to make this into an element in $\widehat{\mathcal L}$. 
 \section{Acknowledgments}
 The author would thank the organizers and audience of the  MSRI conference ``Symplectic and Contact Topology and Dynamics: Puzzles and Horizons
'' and Edinburgh's ICMS Conference ``Symplectic Geometry and Transformation groups'' where these results were first presented in 2010.
 I thank the organizers for the opportunity to present this work. 
N. Vichery used  his joint work with A. Monzner and F. Zapolsky (see \cite{M-V-Z}) to extend the methods and some of the results of this paper to other cotangent bundles in \cite{Vichery3}. 
Finally  the paper has been much improved by the very careful reading and thoughtful suggestions of the referees and I am very grateful for their work. Whatever is left to blame in this paper is of course my own responsibility. 
 \section{Recollection of \texorpdfstring{$\gamma$-convergence}{gamma-convergence}}
 
 Note that $\gamma$-convergence corresponds to the convergence of critical values  obtained by minmax on some homology class for the action functional or rather one of its finite dimensional reductions. In  the case that $H$ is convex in $p$, hence has a Legendre dual $W(t,x,\dot x)$, $\gamma$-convergence implies the convergence of the critical values of the action functional $E(x)=\int_0^1 W(t,x,\dot x) dt$. This can be considered as a generalization of $\Gamma$-convergence, invented by E. de Giorgi (see \cite{DalMaso,Braides}), which in this case would  imply convergence of the minima of the action functional. 
 
 We now shortly summarize the definition of $\gamma$-convergence as well as the main results of symplectic homogenization. 
 
 Let $L$ be a Lagrangian submanifold Hamiltonianly isotopic to the zero section in the cotangent bundle $T^*N$ of a compact manifold $N$. According to \cite{Viterbo-STAGGF}, $L$ has a Generating Function Quadratic at Infinity (G.F.Q.I. for short) $S(q,\xi)$ defined on $N\times {\mathbb R} ^d$, such that
 \begin{enumerate} [label=\theenumi]
 \item the map $(q,\xi) \mapsto \frac{\partial S}{\partial \xi}$  has $0$ as a regular value
 \item $L=\{(q, \frac{\partial S}{\partial q} (q,\xi) \mid \frac{\partial S}{\partial \xi}(x,\xi)=0\}$
 \item There is a non-degenerate quadratic form $Q(\xi)$ such that $S(x,\xi)=Q(\xi)$ for $\vert \xi\vert $ large enough. 
 \end{enumerate} 
 
 This implies that for $c$ large enough, $H^*(S^c,S^{-c})$ (that we shall denote by $H^*(S^{+\infty}, S^{-\infty})$) is isomorphic to $H^{*-d^-}(N)$, where $d^-$ is the dimension of the negative eigenspace of $Q$. So for each class $\theta\in H^*(N)$ we denote by $T(\theta)$  its image in $H^*(S^{+\infty}, S^{-\infty})$, and we set
 
 $$c(\theta,S)=\sup\{ c \mid H^*(S^{+\infty}, S^{-\infty})\longrightarrow H^*(S^{c}, S^{-\infty})\; \text{sends}\; T(\theta)\; \text{to}\; 0\}$$
 
 and 
 
$$ \gamma(S)=c(\mu,S)-c(1,S)$$ where $1\in H^0(N)$ and $\mu_N\in H^n(N)$ is the fundamental class. 

We proved in \cite{Viterbo-STAGGF} that 
\begin{enumerate} [label=\theenumi] 
\item $\gamma (S)$ only depends on $L$. Henceforth we denote it by $\gamma(L)$
\item $\gamma(L)=0$ if and only if $L=0_N$
\end{enumerate} 

Note that  $T^*T^n\times \overline{T^*T^n}$ (where $\overline{T^*T^n}$ is the $T^*T^n$ with the symplectic form $-dp\wedge dq$), has a covering symplectically isomorphic to $T^*(T^n\times {\mathbb R} ^n)$ by the identification

$$(q,p,Q,P) \mapsto (q,P,p-P,Q-q)$$
(to make sense of $q-Q$ we must be on a covering)

As a result the graph of a compact supported Hamiltonian map $\varphi$ in $T^*T^n\times \overline{T^*T^n}$ has a covering that is a Lagrangian submanifold in $T^*(T^n\times {\mathbb R} ^n)$, coinciding with the zero section at infinity (since for  $P$ large enough $\varphi$ is the identity so $P-p=0, q-Q=0$).
We may compactify the graph of $\varphi$ to a Lagrangian in $T^*(T^n\times {\mathbb S}^n)$. This compactification  is denoted by $\Gamma_{\varphi}$ and  we set $\gamma(\varphi)=\gamma(\Gamma_\varphi)$. According to \cite{Viterbo-STAGGF}, we have 
 $$\gamma(\varphi\psi)\leq \gamma(\varphi)+\gamma(\psi)$$
 As a result, $d_\gamma(\varphi,\psi)=\gamma(\varphi\psi^{-1})$ defines an invariant distance on the group $\DHam_c(T^*T^n)$ of compact supported Hamiltonian diffeomorphisms of $T^*T^n$.
 
  The  completion of $({\DHam_c}(T^*T^n), \gamma)$ is a topological group, denoted by $\widehat{\DHam_c}(T^*T^n)$, also called the Humili\`ere completion (see \cite{Humiliere}). Note that any continuous Hamiltonian $H(t,q,p)$ has a flow in   $\widehat{\DHam_c}(T^*T^n)$, but this flow
 does not operate on $T^*T^n$ (i.e. does not induce maps on $T^*T^n$). 
 
 In \cite{STH} we proved  that the sequence $k \mapsto H(k t, k q, p)$ $\gamma$-converges to some continuous Hamiltonian $\overline H(p)$, and if $H$ is Lipschitz, so is $\overline H$. 
  The following summarizes the main results of \cite{STH}, that is theorems 4.1 and 4.2 :
  \begin{theorem}[Symplectic homogenization]  \label{Thm-3.1} \hskip 0pt
Let $H(t,q,p)$ be a compactly supported, $C^{2}$-Hamiltonian, $1$-periodic in $t$ on  $T^*T^n$. Then the following holds:
 \begin{enumerate} [label=\theenumi]
\item \label{3.1-1} There exists a Hamiltonian $\overline H \in C_{c}^{0,1}({\mathbb R}^n, {\mathbb R} )$  such that  the sequence $H_{k} (t,q,p) = H(kt,kq,p)$ $\gamma$-converges to
$\overline{H}(t,q,p)= \overline H(p)$.

\item \label{3.1-2} The function $\overline H$ only depends on $\varphi^1$, the time-one map associated to $H$
(i.e. it does not depend on the isotopy $(\varphi^t)_{t\in [0,1]}$).

\item \label{3.1-3}The map
$$
\mathcal{A} : C_{c}^{2} ([0,1]\times T^{*} T^{n} , \mathbb{R}) \to C_{c}^{0}
(\mathbb{R}^{n},\mathbb{R})
$$
defined by $\mathcal{A}(H) = \overline{H}$ extends to a nonlinear projector (i.e.   satisfying ${\mathcal A}^2= {\mathcal A}$) with Lipschitz constant 1
$$
\mathcal{A} : \widehat{\Ham_c}( T^{*} T^{n})\to
C_{c}^{0}(\mathbb{R}^{n},\mathbb{R})
$$
where $\widehat{\Ham_c}$ is the completion of $C_{c}^{2} ([0,1]\times T^{*} T^{n} , \mathbb{R})$ for the metric  $\gamma$, and the metric on
$C^{0}(\mathbb{R}^{n},\mathbb{R})$ is the $C^{0}$ metric.

\item \label{3.2-1}The map $\mathcal{A}$ is monotone, i.e. if $H_1 \leq H_2$, then ${\mathcal A}(H_1) \leq {\mathcal A}(H_2)$.

 \item \label{3.2-3} It is invariant under the action of a  Hamiltonian symplectomorphism: \newline${\mathcal{A}} (H\circ \psi)= {\mathcal{A}}(H)$ for all $\psi \in {\DHam_c}(T^*T^n)$. 

 \item \label{3.2-5}  We have ${\mathcal A}(-H)=-{\mathcal A}  (H)$. 

\item \label{3.2-2b}  If $L$ is a Lagrangian Hamiltonianly isotopic to $L_{p_{0}}=\{(q,p_{0})\in {\rm T^*T^{n}}\} $ and $H_{\mid L}\geq h$ (resp. $\leq h$) we have ${\mathcal A}(H)(p_{0})\geq h$ (resp. $\leq h$).

 \end{enumerate}
 \end{theorem}

 \section{Aubry-Mather theory for non-convex Hamiltonians}\label{Section-4}
Let us notice that the original version of this paper, from 2010, had a more complicated proof for Theorem \ref{main-theorem}, which relied on Theorem \ref{2.1}.
  
\begin{proof}[Proof of theorem \ref{main-theorem}]  Let  $S$ be a Generating Function Quadratic at Infinity (G.F.Q.I.)  on a product base manifold, $S: X\times Y\times E \longrightarrow {\mathbb R} $, where $E={\mathbb R}^d$. We denote by $S_y$ the restriction of $S$ to $X\times \{y\}\times E$. For $L$ a Lagrangian submanifold in $T^*(X\times Y)$, we denote by $L_y$ the symplectic reduction  $(L\cap (T^*X \times T_y^*Y)) / T_y^*Y$. 
 
 Remember from \cite{STH} that $\overline H$ is the uniform limit of the sequence of functions $(h_k)_{k\geq 1}$  defined as follows. First $\Gamma_k$ is the Lagrangian submanifold in $T^*(T^n\times {\mathbb S}^n)$ defined as the compactification of 
 $$\{(q,P_k(q,p),p-P_k(q,p), Q_k(q,p)-q)\mid \Phi_k(q,p)=(Q_k(q,p),P_k(q,p))\}$$   $\Phi_k=\rho_k^{-1}\Phi^k\rho_k$  where $\rho_k(q,p)=(k\cdot q,p)$. Then we write $\Gamma_{k,y}$ for  $(\Gamma_k)_y$ the reduction of $\Gamma_k$ at $y$ and  we define $h_k(p)=c(\mu_q,\Gamma_{k,p})$, where  $\mu_q$ is the fundamental class in $T^n$.

 Note that we may assume $\alpha \neq 0$, since otherwise we can take $p_k=p$ in a region $\{(q,p) \mid \vert p \vert \geq r\}$ for $r$ large enough, so that $H$ and ${\overline{ H}}$ vanish and the Theorem  is obvious.  
 
The idea of the proof is that at differentiability points for $h_k$, we can construct an invariant measure  by the usual method of renormalized orbits (i.e. the Krylov-Bogolioubov procedure). Then we shall see that at the other points, we can obtain the invariant measure by taking limit of subsequences and convex hulls of the measures corresponding to points of differentiability of the $h_k$. 

 Note that if $S_k(q,P,\xi)$ is a G.F.Q.I. for $\Gamma_k\subset T^*(T^n\times {\mathbb S}^n)$, $h_k(P)$ is defined as the critical value associated to the class $\mu_q$ of the function $(q,\xi) \mapsto S_k(q,P,\xi)$ (see \cite{STH}). If  the selector defining $h_k$ is smooth at $P_k$, i.e. there is a smooth map $P\mapsto (q(P),\xi(P))$ defined in a neighborhood of $P_k$, such that 
 $$ \frac{\partial S_k}{\partial q}(q(P),P,\xi(P))=0= \frac{\partial S_k}{\partial \xi}(q(P),P,\xi(P))$$ and $S_k(q(P),P,\xi(P))=h_k(P)$,  we 
 have $$\alpha_k = dh_k(P_k)=\frac{\partial S_k}{\partial P}(q(P_k),P_k,\xi(P_k))$$ so that the point of $\Gamma_k$ corresponding to 
 $(q(P_k),P_k,\xi(P_k))$ is $(q(P_k), P_k, 0, \alpha_k)$, which translates into $\Phi_k(q(P_k),P_k)=(q(P_k)+\alpha_k, P_k)$, hence $
 \Phi^k(k\cdot q(P_k), P_k)=(k\cdot q(P_k)+k\cdot \alpha_k, P_k)$, and the trajectory $\gamma_k=\{ \varphi^{kt}(q(P_k),P_k) \mid t\in 
 [0,1]\}$ yields a normalized measure $\meas_k$ on $T^*T^n$, such that $\vert (\varphi^1)_*(\meas_k)-\meas_k\vert \leq \frac{2}{k} $  
 (where $\vert \meas-\meas'\vert= \sup_{A\subset T^*T^n} \vert \meas (A)-\meas'(A) \vert$ is the total variation distance of the 
 probabilities $\meas, \meas'$). As a result choosing a subsequence $(P_k)_{k\geq 1}$ such that $\lim_k P_k=p_\infty$,  the limit of $
 \meas_k$ is an invariant measure $\meas$ such that $r(\meas)=\alpha$ and $A(\meas)=\lim_k P_k\cdot \alpha_k-
 h_k(P_k)=p_\infty\cdot \alpha - \overline H(p_\infty)$.

 In general the selector $h_k$ is only smooth outside a closed set of zero measure (see \cite{Ottolenghi-Viterbo} and  Appendix 2, Section \ref{Appendix-11} ) but $h_k$ and $\overline H$ are  Lipschitz since the Lipschitz constant of $h_k$ is bounded by 
 the maximum norm of $Q_k(q,p)-q$ and this is in turn bounded by the maximum of $ \vert \frac{\partial H}{\partial p} \vert $.
   
   Now if $\alpha_k \in \partial_Ch_k(p_k)$, then $(p_k, \alpha_k)$ belongs to ${\rm{LConv}}_{T^n,q}(\Gamma_k)$ a subset of ${\rm Conv}_{T^n,q}(\Gamma_k)$, the fiberwise convex hull of $\Gamma_k$ : we refer to Definition \ref{Definition-10.19}  for the definition of ${\rm{LConv}}_{T^n,q}(\Gamma_k)$ and to  Proposition \ref{Prop-10.17} for the statement. More precisely  ${\rm LConv}_{T^n,q}(\Gamma_k)$ is the set of $(q,P, P-p,q-Q)$ where $q-Q$ is in the convex hull of the set of $q-Q_k(q,p)$ such that $P_k(q,p)=P$ and $\langle P_k, \alpha_k\rangle-h_k(P_k)=A$ is fixed. This means that there are $\alpha_k^j$ such that $(q_j(p_k),p_k, 0, \alpha_k^j)\in \Gamma_k$ and $\alpha_k$ is in the convex hull of the $\alpha_k^j$,  in other words that  $\Phi^k(k\cdot q_j(p_k),p_k)=(k\cdot q_j(p_k)+k\cdot \alpha_k^j,p_k)$. 
 
By Caratehodory's theorem a point in the convex hull of a set in $\mathbb R^n$ is in the convex hull of a subset of cardinality $n+1$. We may thus limit ourselves to $1\leq j \leq n+1$ and by taking subsequences, we can assume that if $\alpha=\lim_k \alpha_k$ we have $\alpha^j= \lim_k \alpha_k^j$, and $\alpha$ is in the convex hull of the $\alpha^j$.
 
This proves that any $\alpha \in \limsup_k \partial_C h_k(p_k)$ is the rotation number of an invariant  measure for $\varphi^1$. 
Its action is obtained by noticing that  $\Phi^k(k\cdot q_j(p_k),p_k)=(k\cdot q_j(p_k)+k\cdot \alpha_k^j,p_k)$ corresponds to the intersection of $\Gamma_k$ and the Lagrangian $\Lambda_{\alpha_k^j}= \{(q,P, 0, \alpha_k^j) \mid q\in T^n, P\in {\mathbb R}^n \}$. Note that  $\Lambda_\alpha$ has generating function $\langle P, \alpha\rangle$. So the intersection is given by the critical points of $S_k^\alpha(q,P;\xi)=\langle P, \alpha\rangle - S_k(q,P;\xi)$, and then $c(\mu_q, S_k^\alpha)= \langle p_k, \alpha \rangle-h_k(p_k)$. Here the sign of $S_k^\alpha$ comes from the fact that on $T^* {\mathbb R}^n \times \overline{T^* {\mathbb R}^n}$ the symplectic form  is $dp\wedge dq-dP\wedge dQ$ and not $dP\wedge dQ-dp\wedge dq$. As a result, setting $p_\infty=\lim_kp_k$,  the limit of the critical value as $k$ goes  to infinity will be  $\langle p_\infty,\alpha \rangle - \overline H(p_\infty)$.  

 Now given $\alpha \in \partial_C{\overline H}(p_\infty)$, 
it follows from  \cite{Jourani} (thm 3.2, using the fact that by thm 2.1, ibid, for Lipschitz functions,  $\partial_Af=\partial_Gf$),   that, up to taking subsequences, we may assume,   $\lim_k p_k=p_\infty$, we have $$\partial_C{\overline H}(p_\infty)\subset  \limsup_k \partial_C h_k(p_k)$$ (see  Appendix \ref{Appendix-crit}, Prop\ref{Prop-12.4}).
 Thus we have  sequences $\alpha_k^j$ in $\partial_Ch_k(p_k)$ such that $\alpha$ is in the convex hull of the $\alpha_\infty^j$ where $\alpha_\infty^j=\lim_k \alpha_k^j$ up to taking subsequences. 
  Setting  $\gamma_k^j=\{ \varphi^{kt}(q_j(P),P) \mid t\in [0,1]\}$, then  $\frac{1}{k}[\gamma^j_k]$ -the probability measure that is the direct image by $t\mapsto \varphi^t(q_j(P),P)$ of the normalized Lebesgue measure on $[0,k]$-  converges  to the probability measure $\meas_j$, with rotation vector $\alpha_j$ and action $\langle p_\infty,\alpha_j\rangle - \overline H(p_\infty)$, so the action of the convex hull of these measures contains a measure with rotation vector $\alpha$ and action $\langle p_\infty, \alpha\rangle - \overline H(p_\infty)$.
  \end{proof}

 \section{Strong convergence in Symplectic homogenization}
 
 Let $S$ be a \GFQI for the Lagrangian $L$ in $T^*N$. Provided $L$ is Hamiltonianly isotopic to the zero section, we know that up to equivalence $S$ is unique (see  \cite{Theret, Viterbo-STAGGF}), hence the Generating Function homology defined in \cite{Traynor} as $$GH_*(L; a, b) =H_{*-i}(S^b, S^a)$$ (resp. cohomology $GH^*(L; a, b) =H^{*-i}(S^b, S^a)$)
 where $i$ is the index of the quadratic form defined by $S$ at infinity, is well defined and only depends on $L$, not on the choice of $S$. It was proved in \cite{Viterbo-FCFH2} that Generating Function homology (resp. cohomology) coincides with the Floer homology $FH_*(L,0_N;a,b)$ (resp. $FH^*(L,0_N;a,b)$) defined in \cite{Floer-Morse-index, Floer1, Floer2, Floer-Witten}. We shall from now on use the notation $FH_*$ (resp. $FH^*$) even though we shall often consider it as Generating function homology (resp. cohomology).  
 
 Similarly if $L_1,L_2$ have \GFQI $S_1, S_2$, we set $(S_1\ominus S_2)(x;\xi,\eta)=S_1(x;\xi)-S_2(x;\eta)$ and $$FH_*(L_1,L_2;a,b)=H_{*-i}((S_\ominus S_2)^b, (S_\ominus S_2)^a).$$
 In the sequel we shall omit the grading shift by $i$. 
 
 The goal of this section is to improve the convergence result of \cite{STH}. Indeed, we require that for a sequence $S_k$ of \GFQIs  for $\varphi_k(L)$, not only the critical values of $S_k$ corresponding to certain homology classes converge, but that all ``homological critical values" (that is $c$ such that $\lim_{ \varepsilon \to 0}H_*(S_k^{c+ \varepsilon }, S_k^{ c- \varepsilon} )\neq 0$) converge to ``homological critical values'' of $\overline S$, the \GFQI for $\overline {\varphi} (L)$. This is what we mean by $h$-convergence, and proving $h$-convergence of $\varphi_k(L)$ to $\overline \varphi (L)$ will yield a sequence of closed orbits of rotation number $\alpha_k$ converging to $\alpha$ for all $\alpha$ in $d_s\overline H (p)$ (see Definition \ref{Def-9.14}), and such that the actions also converge. This will be explained in detail in section \ref{section-6}. 
 
 Remember from Section \ref{Section-4} that we defined the sequence $\varphi_k=\rho_k^{-1}\varphi^k\rho_k$ where $\rho_k(q,p)=(k\cdot q,p)$ and even though  $\rho_k^{-1}$ is not well-defined, but $\varphi_k$ is well defined if $\varphi$ is Hamiltonianly isotopic to the identity, as the unique solution of $\rho_k\varphi_k=\varphi^k\rho_k$ obtained by continuation starting from $\varphi=\varphi_k=\Id$.

We shall now prove the following theorem, which is a refinement  of $\gamma$-convergence
 
 \begin{theorem}   \label{2.1}
Let $a<b$ be real numbers,  $L_{1},L_{2}$ be Lagrangian submanifolds Hamiltonianly isotopic to the zero section. There exist a sequence $( \varepsilon _{k})_{k \geq 1}$ converging to zero, and  maps $$ i_{k}^{a,b}: FH_{*}(\varphi_{k}(L_{1}),L_{2}; a,b) \longrightarrow FH_{*}(\overline\varphi (L_{1}),L_{2}; a+ \varepsilon _{k},b+  \varepsilon _{k})$$ and $$j_{k}^{a,b}:FH_{*}(\overline\varphi (L_{1}),L_{2}; a,b) \longrightarrow FH_{*}(\varphi_{k} (L_{1}),L_{2};a+ \varepsilon _{k},b+ \varepsilon _{k})$$ such that the maps $$i_{k}^{a+ \varepsilon _k,b+ \varepsilon _k}\circ j_{k}^{a,b}: FH_{*}(\overline\varphi  (L_1),L_2; a,b)  \longrightarrow  FH_{*}(\overline\varphi (L_1),L_2 ; a+ 2\varepsilon _{k},b+ 2\varepsilon _{k})$$ and 
 $$j_{k}^{a+ \varepsilon _k,b+ \varepsilon _k}\circ i_{k}^{a,b}: FH_{*}(\varphi_k(L_1),L_2; a,b)  \longrightarrow  FH_{*}(\varphi_k (L_1),L_2 ; a+ 2\varepsilon _{k},b+ 2\varepsilon _{k})$$
converges to the identity as $k$ goes to infinity.
Moreover the maps $i_{k}^{a,b}, j_{k}^{a,b}$ are natural, that is the following diagrams are commutative for $a<b$, $c<d$ satisfying $a<c, b<d$

$$
\xymatrix
{FH_*(\varphi_{k} (L_{1}),L_{2}; c ,d)\ar[d] \ar[r]^-{i_{k}^{c,d}}&
FH_*(\overline\varphi (L_{1}),L_{2}; c+ \varepsilon _{k},d+ \varepsilon _{k})\ar[d] \\
FH_*(\varphi_{k} (L_{1}),L_{2}; a ,b)\ar[r]^-{i_{k}^{a,b}} &  FH_*(\overline\varphi (L_{1}),L_{2}; a+ \varepsilon_{k},b+ \varepsilon _{k})  }
$$

$$
\xymatrix
{
FH_*(\overline\varphi (L_1),L_2; c,d)\ar[d] \ar^-{j_{k}^{c,d}}[r]& FH_*(\varphi_{k} (L_1),L_2; c+ \varepsilon_k ,d+ \varepsilon _k)\ar[d]  \\
 FH_*(\overline\varphi (L_1),L_2; a,b) \ar[r]^-{j_{k}^{a,b}} & FH_*(\varphi_{k} (L_{1}),L_{2}; a+ \varepsilon _{k},b+ \varepsilon _{k})  }
$$
where the vertical maps are the natural maps. 
\end{theorem} 

\begin{remarks} \label{2.2} \begin{enumerate} [label=\theenumi]
 \item
Note that $\overline\varphi(L)$ is the limit of $\varphi_k(L)$ in the completion $\widehat {\mathcal L}$. The fact that $\varphi_k$ converges to $\overline \varphi$ in $\widehat \DHam(T^*T^n)$ implies easily that $\varphi_k(L)$ has a limit in $\widehat{\mathcal L}$ denoted $\overline\varphi(L)$. 

In this situation, we proved in  \cite{Viterbo-FCFH1} that  $FH^*(\overline\varphi (L);a,b)$ is well defined. We shall be more precise about that in section \ref{Integrable-Hamiltonians}.

\item \label{2.2.c} We will show that the maps $$i_{k}^{a+ \varepsilon _k,b+ \varepsilon _k}\circ j_{k}^{a,b}: FH_{*}(\overline\varphi  (L_1),L_2; a,b)  \longrightarrow  FH_{*}(\overline\varphi (L_1),L_2 ; a+ 2\varepsilon _{k},b+ 2\varepsilon _{k})$$ and 
$$j_{k}^{a+ \varepsilon _k,b+ \varepsilon _k}\circ i_{k}^{a,b}: FH_{*}(\varphi_k(L_1),L_2; a,b)  \longrightarrow  FH_{*}(\varphi_k (L_1),L_2 ; a+ 2\varepsilon _{k},b+ 2\varepsilon _{k})$$
 coincide with the natural map induced by the inclusions of the sublevel sets (or chain complexes). Therefore they obviously converge to the identity as $ \varepsilon_k$ goes to $0$. 
\item There is of course by Poincar\'e duality a similar result in cohomology, except the arrows are going in the opposite direction. 
\item \label{2.2.b} The results in \cite{STH} imply that $\varphi_{k}\times \Id $ $\gamma$-converges to $\overline \varphi \times \Id$, so if $L$ is the graph of a Hamiltonian map, $\psi$, we get that the result in the Theorem still holds with $\varphi_{k}(L)$ and $\overline \varphi (L)$ replaced by $\varphi_{k}\psi$  and $\overline \varphi \psi$. 
It is easy to apply the above definition to prove that $h$-convergence implies $\gamma$-convergence. Indeed, for each class $\alpha$ we have that $h$-convergence of $L_k$ to $\overline{L}$ implies that $c(\alpha, L_k, \overline L)$ goes to zero. As a consequence $c(\alpha, \varphi_k\overline \varphi^{-1}) \longrightarrow 0$ and this implies $\gamma$ convergence of $\varphi_k$ to $\overline \varphi$. 
\end{enumerate} 
\end{remarks}

To prove  Theorem \ref{2.1}, we shall use Lisa Traynor's version of Floer homology as Generating function homology (see \cite{Traynor} for the definition and \cite{Viterbo-FCFH2} for the proof of the isomorphism between Floer and Generating Homology).
This means that if $S$ is a \GFQI for $L$ then up to a shift in grading, we have $H^*(S^b, S^a)=FH^*(L;a,b)$. The left-hand side is Generating Function cohomology and according to a result by Th\'eret and the author does not depend on the choice of $S$. The  reader could take this as a definition of Floer cohomology in our setting, but we shall occasionally use properties of Generating Function cohomology that are stated in the literature in the framework of  Floer cohomology.

 We will in fact compare the relative homology 
of generating functions corresponding to $\overline \varphi (L)$ and $\varphi_{k}(L)$. 
 As in \cite{STH} we consider   $S_{1}(x,\eta_{1})$, $S_{2}(x,\eta_{2})$, GFQI respectively for  $L_{1}$ and $L_{2}$, and  $F_{k}(x,y,\xi)$ a GFQI for $\varphi_{k}$ so that  $$\varphi_{k}\left (x+\frac{\partial F_{k}}{\partial y} (x,y,\xi) ,y \right )=\left (x,y +\frac{\partial F_{k}}{\partial x}(x,y,\xi)\right ) \Leftrightarrow \frac{\partial F_{k}}{\partial \xi}(x,y,\xi)=0$$
 and then  we obtain a generating function of $\varphi_{k}(L_{1})-L_{2}$ $$G_{k}(x;y,u,\xi, \eta)=S_{1}(u;\eta_1)+F_{k}(x,y,\xi)+ \langle y,x-u\rangle -S_{2}(x,\eta_{2}).$$ 
 where $\eta=(\eta_1,\eta_2)$.
 Similarly if $h_{k}(y)=c(\mu_{x},F_{k,y})$ and $\overline\varphi_{k}$ the flow of the integrable Hamiltonian $h_{k}$ we have the following ``generating function'' of $\overline\varphi_{k}(L_{1})-L_{2}$
  $$\overline G_{k}(x;y, u,\eta)=S_{1}(u;\eta_{1})+h_{k}(y)+ \langle y,x-u\rangle -S_{2}(x; \eta_{2})$$

We will also set ($k=+\infty$)
 $$\overline G_{}(x;y, u,\eta)=S_{1}(u;\eta_{1})+\overline H(y)+ \langle y,x-u\rangle -S_{2}(x; \eta_{2}).$$
and this can be considered as a ``sui generis'' generating function for $\overline\varphi(L)$, the important fact is that by $FH_*(\overline\varphi(L_1),L_2;a,b)$ we mean
$H_{*-i}(\overline G;a,b)$.  

First of all we shall use the explicit form for $F_k$ given by 
\begin{definition}\label{4.5} Let $S$ be a \GFQI for $\varphi$. We then set 
 \begin{gather*}  \label{def-F} F_k(x,y;\xi, \zeta)= \frac{1}{k}\left [ S(kx,p_{1}; \zeta_1)+ \sum_{j=2}^{k-1} S(kq_{j},p_{j}; \zeta_j) + S(kq_{k},y;\zeta_k)\right ]+B_{k}(x,y,\xi)\end{gather*} where $\zeta = (\zeta_1, ..., \zeta_k)$ and
\begin{gather*} 
B_{k} (q_{1}, p_{k}; p_{1} , q_{2} ,\cdots, q_{k-1} , p_{k-1} ,
q_{k}) = B_k(q_1,p_k;\xi)=\\
\sum_{j=1}^{k-1} \langle p_{j}, q_{j+1} - q_{j} \rangle +
\langle p_{k} , q_{1}-q_{k}\rangle \ .
\end{gather*} 
We then define $E_k^\pm$ to be the positive and negative eigenspace of the quadratic form asymptotic to $F_k$ at infinity. We finally  set ${h}_{k} (y) = c (\mu_{x} , {F}_{k,y})$ where
${F}_{k,y}(x;\xi, \zeta) = {F}_{k} (x,y; \xi, \zeta)$. 
\end{definition}

In the next lemma we assume we are given open sets  $V_j^\delta\subset U_j^\delta$ in $ {\mathbb R}^n$ for $j\in [1,\ell]$ such that 
\begin{enumerate} [label=\theenumi] 
\item the diameter of the connected components of the complement of $V_j^\delta$ is less than $\delta$
\item the intersection of any $n+1$ distinct $U_j^\delta$ is empty
\end{enumerate} 
The construction of the $U_j^\delta$ is easy and can be found in \cite{STH}, proof of proposition 5.3.
We define $\chi_j^\delta$ to be smooth functions supported in $U_j^\delta$ and equal to $1$ on $V_j^\delta$.
In  the rest of the paper, $H_*$ will designate singular homology with coefficients in some field. 

 Since $h_{k}(y)=c(\mu_{x} ,{F}_{k,y})$, this means there exists a cycle ${\widetilde C}^-(y)$ with $[{\widetilde C}^-(y)]=[T^{n}_{x}\times E_{k}^-]$ in $H_{*}(F_{k,y}^\infty,F_{k,y}^{-\infty})$, and 
  $$h_{k}(y)- \varepsilon \leq \sup_{(x,\xi)\in {\widetilde C}^-(y)}F_{k}(x,y,\xi) \leq h_{k}(y)+ \varepsilon$$
  
  Unfortunately we cannot choose ${\widetilde C}^-(y)$ to depend continuously on $y$.

The next Proposition will yield the existence of a continuous family $C^{-}(y)$ that coincides with ${\widetilde C}^-(y)$ outside a small set of $y$ given by one of the $U_j^\delta$. 
\begin{proposition} 
Let $\overline \Gamma$ be a cycle in $H_{*}(\overline G_{k}^b, \overline G_{k}^a)$. Then there exists a sequence $ \varepsilon _{k}$ of positive numbers converging to $0$ and a continuous family of cycles $y \mapsto C^-(y)$, where $C^{-}(y)$ is a cycle homologous to $T^n_{x}\times \{y\}\times E_{k}^-$ in $H_*(F_{k,y}^{+\infty},F_{k,y}^{-\infty})$  such that for some constant $M$ we have 
$$F_{k}(x,y,\xi, \zeta) \leq h_{k}(y)+M\chi^j_{\delta}(y)+ \varepsilon_{k}$$ whenever $(x,\xi, \zeta)\in C^-(y)$.

\end{proposition} 

\begin{proof} 
  
We remind the reader of the Definition\footnote{We shall extend this definition in Appendix \ref{Appendix-crit}, Definition \ref{Definition-PS2} for the case where $F$ is only Lipschitz.}
\begin{definition} A smooth function $F$ on $X$ satisfies the Palais-Smale condition ((PS) condition for short) if any sequence $(x_n)_{n\geq 1}$  such that $F(x_{n})$ is bounded and $\nabla F(x_{n})$ converges to zero has a converging subsequence.
\end{definition} 
 This implies that the flow of $\frac{\nabla F}{ \vert \nabla F (x) \vert^{2} }$
is defined for all times outside a neighborhood of the critical points (which moreover form a compact set). In particular, classes in $H^{*}(F^{b}, F^{a})$ are represented by linear combinations of unstable manifolds of critical points (see for example \cite{Laudenbach-Bismut}). As usual we denote by $F^{\infty}$ (resp. $F^{-\infty}$) the set $F^{c}=\{ x \mid F(x) \leq c \}$ for $c$ large.  
The Proposition is then an obvious consequence of the following
\begin{lemma}
Let $F(u,x)$ be a smooth function on the product  $V\times X$ of two oriented manifolds. We moreover assume both $F$ and its restriction $F_{u}$ to a fiber $\{u\}\times X$ satisfy the Palais-Smale condition and $F_{u}$ does not depend on $u$ for $u$ outside a compact set. 

Let  $f\in C^0(V, {\mathbb R} )$ be such that  for each $u \in V$, there exists a cycle $C(u)$ representing a class in $ H_{d}(F_{u}^{+\infty},F_{u}^{-\infty})$
with $F(u, C(u))\leq f(u)$.  
Moreover we assume that $H_{k}(F_{u}^{\infty}, F_{u}^{-\infty})$ vanishes for $k\geq d+1$. 

Then for any positive $ \varepsilon $ and any subset $U$ in $V$, such that each connected component of $V\setminus U$ has sufficiently small  diameter, there exists a cycle $\widetilde C$ in $H_{d+\dim(V)}(F^{\infty}, F^{-\infty})$ and a constant $M$ such that if we denote by $\widetilde{C}(u)$ the slice $\widetilde C \cap \pi^{-1}(u)$ ($\pi:V\times X \longrightarrow X$ is the second projection)
$$F(u,{\widetilde C}(u)) \leq f(u)+ M \chi_{U}(u)+ \varepsilon $$ where $\chi_U$ is the characteristic function of $U$.
\end{lemma}

\begin{proof} This is lemma 5.1 in \cite{STH}.
\end{proof} 
\end{proof} 
\begin{proof} [Proof of Theorem \ref{2.1}]

  We thus return to our original problem, and consider a continuous family of cycles ${ C}^-(y)$ but now the inequality 
  $$h_{k}(y)- \varepsilon \leq \sup_{(u,\xi)\in C^-(y)}F_{k}(u,y,\xi) \leq h_{k}(y)+ \varepsilon$$
  only holds outside a set $U_{2\delta}$, where $U_{\delta}$ is a neighborhood of a fine grid $ \varepsilon {\mathbb Z}^n$ in $ {\mathbb R} ^n$, while  we have the general bound

  $$ \left \vert \sup_{(u,\xi)\in C^-(y)}F_{k}(u,y,\xi) -h_{k}(y)\right \vert \leq M \chi_{\delta}(y) + \varepsilon $$ where $\chi_{\delta}$ is $1$ in $U_{\delta}$ and vanishes outside $U_{2\delta}$.  
  
We consider $\ell$ different such continuous families, corresponding to function $\chi_j^\delta$, and remember that their supports $U_j^\delta$ satisfy (see Figure 1 in \cite{STH})
  
\begin{enumerate} [label=\theenumi] \label{grid}
\item the connected components of  ${\mathbb R}^n\setminus U_{j}^{\delta}$, the complement of  $U_{j}^{\delta}$, have diameter less than $\delta$
\item any $(n+1)$ distinct $U_{j}^{\delta}$ have empty intersection. 
\end{enumerate} 
Of course $\delta$ depends on $\ell$, and in the sequel we always assume $\delta$ is chosen appropriately.

  We can then use $F_{k}$ to write a generating function for $\Phi_{\ell k}(L_{1})-L_{2}$ (see \cite{STH}): 
  
\begin{gather*} 
G_{\ell,k} (x_{1};v, \overline x , \overline y, \overline\xi , \overline \zeta, \eta) =\\
{ S}_{1}(u, \eta_{1}) +
\frac{1}{\ell} \sum_{j=1}^{\ell} {F}_{k} (\ell x_{j} , y_{j} ,
\xi_{j}, \zeta_j ) + Q_{\ell} (\overline x ,\overline y)
+ \langle y_{\ell}-v , u - x_{1} \rangle -S_{2}(x_{1},\eta_{2})
 \end{gather*} 
where $\eta= (\eta_1, \eta_2)$, $\overline \xi = (\xi_1,..., \xi_\ell), \overline \zeta = (\zeta_1,..., \zeta_\ell)$ and 
$$Q_{\ell} (\overline x ,\overline y)=
B_{\ell} (x_{1}, y_{k}; y_{1} , x_{2} ,\cdots, x_{\ell-1} , y_{\ell-1} ,
x_{\ell}) =
\sum_{j=1}^{\ell-1} \langle y_{j}, x_{j+1} - x_{j} \rangle +
\langle y_{\ell} , x_{1}-x_{\ell}\rangle
$$

We then consider 

\begin{gather*} \overline G_{\ell,k} (x_{1}, u; \overline x , \overline y , \eta) =\\
{ S}_{1}(u, \eta_{1})+
\frac{1}{\ell} \sum_{j=1}^{\ell} ({h}_{k} ( y_{j}) + M \chi_{j}^\delta (y_{j}) )+ Q_{\ell} (\overline x ,\overline y)
+ \langle y_{\ell}-v , u - x_{1} \rangle -S_{2}(x_{1},\eta_{2})
\end{gather*} 

From now on we shall assume $ \varepsilon\ll b-a$.
Let $\Gamma$ be a cycle in a nonzero homology class in   $H^*(\overline G_{\ell,k}^b, \overline G_{\ell,k}^{a})$, 
and consider the cycle \begin{gather*} (\Gamma\times_{Y}  C^-[\ell])= \left\{ (u ;\overline x, \overline y ,\overline\xi, \overline \zeta , \eta_{1},\eta_{2})
\mid (u , \overline x , \overline y, \eta) \in \Gamma , (\ell x_{j}
,\xi_{j}, \zeta_j) \in \tilde C_{j}^- (y_{j})\right\}\ .
\end{gather*}

 It is contained in  $ G_{\ell,k}^{b+ \varepsilon }$, and its boundary is in $ G_{\ell,k}^{a+ \varepsilon }$. It thus represents a class in $H_*(G_{\ell, k}^{b+ \varepsilon }, G_{\ell, k}^{a+ \varepsilon })$.

We still have to identify the limit as $k,\ell$ go to infinity of  $H_*(\overline G_{\ell,k}^{b}, \overline G_{\ell,k}^{a})$ with $H_{*}(\overline G^b, \overline G^{a})$.

Let 
$$K_{\ell,k}^{ }= \frac{1}{\ell} \left( \sum_{j=1}^{\ell} {h}_{k}(y) + a_{k}
  \chi^{\delta}_{j} (y)\right) \ .
$$ be a Hamiltonian and  $\overline\Psi_{k,\ell}$ be its flow. Clearly $\overline G_{\ell,k}$ is a generating function for $\overline\Psi_{k,\ell}(L_{1})-L_{2}$. 

Now at most $(n+1)$ of the supports of $\chi_{j}^{\delta}$ intersect, so that $$\vert K_{\ell,k}^{}(y)-h_{k}(y) \vert  \leq \frac{A}{ \ell}$$
and this difference goes to zero as $\ell$ goes to infinity and since  $h_{k}(y)$ converges to $\overline H (y)$ (see \cite{STH}, Lemma 4.11 and  Proposition 4.12), we have for $k,\ell$ large enough 
$$\vert K_{\ell,k}^{}(y)-\overline H(y) \vert \leq \varepsilon _{k,\ell} $$
where for each $k$, the sequence $\varepsilon_{k,l}$ converges to $0$. 

This classically entails the existence of maps (\cite{Viterbo-FCFH1}, proposition 1.1 and Remark 1.2) which on the lower line are induced by inclusions of the sublevel sets. 
 $$
 \xymatrix{ 
 FH_{*}(\overline \varphi (L_{1}),L_{2}; a,b)\ar[rr]^{i_k^{a,b}}\ar[d]^{\simeq}&&   FH_{*}(\overline\Psi_{\ell,k}(L_{1}), L_{2}; a+ \varepsilon ,b+ \varepsilon  ) \ar[d]^{\simeq}\\
 H_{*}(\overline G^b, \overline G^a)\ar[rr]&&H_*(\overline G_{\ell,k}^{b+ \varepsilon }, \overline G_{\ell,k}^{a+ \varepsilon })
 } $$
 and 
 $$
 \xymatrix{FH_{*}(\overline \varphi (L_{1}),L_{2}; a,b) \ar[rr]^{j_k^{a,b}}\ar[d]^{\simeq}&&   FH_{*}(\overline\Psi_{\ell,k}(L_{1}), L_{2}; a+ \varepsilon ,b+ \varepsilon  )\ar[d]^{\simeq}\\H_{*}(\overline G^b, \overline G^a) \ar[rr]&& H_*(\overline G_{\ell,k}^{b+ \varepsilon }, \overline G_{\ell,k}^{a+ \varepsilon }) 
 }$$
 such that the composition are the canonical maps 
 
$$FH_{*}(\overline \varphi (L_{1}),L_{2}, a,b) \longrightarrow  FH_{*}(\overline \varphi (L_{1}),L_{2}, a+2 \varepsilon_k,b+2 \varepsilon_k)$$
$$FH_{*}( \Phi_{k\ell} (L_{1}),L_{2}; a ,b ) \longrightarrow FH_{*}( \Phi_{k\ell} (L_{1}),L_{2}; a+ 2 \varepsilon_k ,b+ 2 \varepsilon_k )$$ This concludes the constructions.

\end{proof} 

\begin{remark} The family of Floer cohomology vector spaces $t \mapsto  FH^*(L; t)$ defines a persistence module in the sense of \cite{E-L-Z} (see also \cite{Barannikov} and \cite{Z-C, LNV, PolShel, L-S-V}). One can then define a distance between persistence modules, the `` interleaving distance'' and Theorem \ref{2.1} states that the $FH^*(\varphi_k^1(L), t)$ converge (for the interleaving distance) to $FH^*(\overline{\varphi}^1(L), t)$
In terms of barcodes,  this means that the barcodes of $\varphi_k^1(L)$ converge to the barcode of $\overline\varphi^1$. 
\end{remark}

\section{Floer cohomology for \texorpdfstring{$C^0$}{C0}  integrable Hamiltonians}\label{Integrable-Hamiltonians}

 Let $H(p)$ be a smooth integrable Hamiltonian. Then the corresponding flow is 
$ (q,p) \mapsto (q+t\nabla H(p), p)$. If we consider its graph $\{(q,p,Q,P) \mid (Q,P)=\varphi^1(q,p)\}$ and its image by 
$(q,p,Q,P) \mapsto ( q,P, P-p, q-Q)$ is 
$\Gamma_H=\{(q,p, 0, \nabla H(p))\}$ and has $S(x,y)=H(y)$ as generating function (with no fibre variable). 
In the rest of the section we assume $H$ is a compact supported $C^0$ Hamiltonian, $\varphi^t$ its flow in $\widehat{\DHam}_c(T^*T^n)$. We can think of $H$ as the limit of smooth Hamiltonians $H_k$ depending only on $p$ and having a common compact support. Then $\varphi^t$ is the limit of the $\varphi_k^t$ and $\Gamma_H$ the limit of the $\Gamma_{H_k}$ in the completion of the space of Lagrangians in $T^*(T^n\times {\mathbb S}^n)$. Note that since the generating function for 
$\Gamma_{H_k}$ is $S_k(q,P)=H_k(P)$, we have that $S_k$ $C^0$-converges to $S(q,P)=H(P)$.

In the following proposition, we refer to Appendix \ref{Appendix-crit}, Definition \ref{Def-9.14} for the definition of  $d_s$. Note that according to Proposition \ref{near-strong}, if $\alpha \in d_sH(p)$ then $\alpha \in d_sH_k(p_k)$ for some sequence $(p_k)_{k\geq 1}$. 

We denote by $H_\alpha(p)=H(p)-\langle \alpha, p \rangle$.

\begin{proposition} 
Let $\alpha \in d_sH(p)$ and $c=H(p)$, we have for $ \varepsilon >0 $ small enough  $$FH^*(\Gamma_{H_\alpha};c+ \varepsilon , c- \varepsilon )\overset{def}{=}FH^*(\Gamma_{H_\alpha}, 0_{T^n\times {\mathbb R}^n};c+ \varepsilon , c- \varepsilon ) \neq 0$$
\end{proposition} 
\begin{proof} 
Indeed, since $S(x,y)=H(y)$ is a generating function for $\Gamma_H$, $S_\alpha (x,y)=H(y)-\langle y,\alpha\rangle$ is a generating function for $\Gamma_{H_\alpha}$. So we have 
$FH^*(\Gamma_{H_\alpha}, 0_{T^n\times {\mathbb R}^n};c+ \varepsilon , c- \varepsilon ) = H^*(H_\alpha^{c+ \varepsilon }, H_\alpha^{c- \varepsilon })$. By Definition \ref{Def-9.14} this is non-zero if $\alpha \in d_sH(p)$. 
\end{proof} 

Let  $\alpha \in {\mathbb R}^n$  and set   $\Lambda_{\alpha}=\{ (x,y,X,Y) \mid X=0, Y=\alpha\}$. Now let $f_{\alpha,C}(x,y)=\langle \alpha , y\rangle\chi ( \frac{y}{C})$ where $\chi(y)=1$ for $ \vert y \vert \leq 1$, and vanishes for $ \vert y \vert \geq 2$. Then  $ \frac{\partial f_{\alpha,C}}{\partial x}(x,y)=0$ and $\frac{\partial f_{\alpha,C}}{\partial y}(x,y))=\alpha$ for $\vert y \vert \leq C$, so  $\tilde H_\alpha(y) = H(y)-f_{\alpha,C}(x,y)$ coincides with $H_\alpha$ in $\{(x,y,\xi,\eta) \mid \vert y \vert \leq C\}$ and
   $$\widetilde\Lambda_{\alpha}= \{(x,y, \frac{\partial f_{u,C}}{\partial x}(x,y), \frac{\partial f_{\alpha,C}}{\partial y}(x,y))\mid (x,y)\in T^n\times {\mathbb S} ^n\}$$ coincides with $\Lambda_\alpha$ in $\{(x,y,\xi,\eta) \mid \vert y \vert \leq C\}$, so $\Lambda_\alpha \cap \Gamma_H = \widetilde\Lambda_\alpha \cap \Gamma_H$, provided the Lipschitz constant of $H$ is less than $C$ and both $\Gamma_H$ and $\widetilde\Lambda_\alpha$ are compact supported.  
 Provided $C$ is chosen to be large enough, the following is a consequence of the above remarks :

\begin{Cor}\label{Cor-8.2} If $\alpha \in {\mathbb R}^n$ is such that $\alpha \in d_sH(p)$, and $C$ is large enough,  then for $c= H(p)$, 
\begin{gather*} FH^*(\Gamma_H, \Lambda_\alpha, c+ \varepsilon , c-\varepsilon ) =FH^*(\Gamma_H, \widetilde\Lambda_\alpha, c+ \varepsilon , c-\varepsilon)= \\ 
FH^*(\Gamma_{H_\alpha}, 0_{T^n\times {\mathbb R}^n};c+ \varepsilon , c- \varepsilon )=FH^*(\Gamma_{H_\alpha}; c+ \varepsilon , c- \varepsilon ) \neq 0
\end{gather*} 
\end{Cor} 
\begin{proof} 
Indeed, for the first equality, we may deform $\Lambda_\alpha$ to $\widetilde \Lambda_\alpha$ in such a way that we do not introduce any new intersection point  : the deformation is achieved by moving $C$ to $+\infty$ starting from $C$ larger than the Lipschitz constant of $H$. For the second one, since $\widetilde \Lambda_\alpha$ is the graph of the differential of $f_{\alpha,C}$, we have 
\begin{gather*} FH^*(\Gamma_H, \widetilde\Lambda_\alpha; c+ \varepsilon , c-\varepsilon)=FH^*(\Gamma_H, gr(df_{\alpha,C}); c+ \varepsilon , c-\varepsilon)=\\ FH^*(\Gamma_H- gr(df_{\alpha,C}),0_{T^n\times {\mathbb R}^n}; c+ \varepsilon , c-\varepsilon)= FH^*(\Gamma_{H-f_{\alpha,C}},0_{T^n\times {\mathbb R}^n}; c+ \varepsilon , c-\varepsilon)
\end{gather*} 
Deforming $C$ to $+\infty$ we see that the last term is equal to $$FH^*(\Gamma_{H_{\alpha}},0_{T^n\times {\mathbb R}^n} ; c+ \varepsilon , c-\varepsilon)$$
\end{proof}

 \section{A proof of  the weak  Conley conjecture on \texorpdfstring{$T^*T^n$}{T*Tn}}\label{section-6}
Let us consider now the case of periodic orbits and prove Theorem \ref{Thm-weak-Conley}. We shall start by proving  the first part of Theorem \ref{Thm-weak-Conley}
\begin{proposition} \label{Prop-9.1}
Let $\alpha=\frac{u}{v}$ with $u\in \ZZ^n, v\in {\mathbb N}^*$ be a  rational vector written in irreducible form\footnote{i.e. $ \vert v \vert $ is minimal} such that $\alpha \in d_s\overline H (p)$. 
Then there exists a periodic orbit for $\varphi$ with rotation vector $\alpha$, that is $\Phi^{kv}(q,p)=(q+ku,p)$ and average action $\langle p, \alpha \rangle - \overline H(p)$. 
\end{proposition} 
\begin{proof}
 Let $\alpha$ be a rational vector. We write $\alpha= \frac{u}{v}$ with $u\in \ZZ^n, v\in {\mathbb N}^*$ mutually prime.   We need to find fixed points of $\Phi^{kv}-ku$ that will yield periodic orbits of $\Phi$ of period $kv$ and rotation vector $ \frac{u}{v}$. This is equivalent to finding fixed points of $\rho_{k}^{-1}\Phi^{kv}\rho_{k}-u= \Phi_k^v-u$. 
  
   If $\Gamma_{k}^v$ is the graph of $\Phi_{k}^v$ that is $$\Gamma_{k}^v=\{
  ( q, P_{k}(q,p), P_{k}(q,p)-p, q-Q_{k}(q,p))\mid (Q_{k},P_{k})=\Phi_{k}^v(q,p)\}$$
  
  and we look for points in $\Gamma_{k}^v \cap \Lambda_{u}$ where $\Lambda_{u}=\{ (x,y,X,Y) \mid X=0, Y=u\}$, as before, $f_{u,C}(x,y)=\langle u , y\rangle\chi ( \frac{y}{C})$ where $\chi=1$ for $ \vert y \vert \leq 1$, and vanishes for $ \vert y \vert \geq 2$, and   $$\widetilde\Lambda_{u}= \{(x,y, \frac{\partial f_{u,C}}{\partial x}(x,y), \frac{\partial f_{u,C}}{\partial y}(x,y))\mid (x,y)\in T^n\times {\mathbb S} ^n\}$$
  
 We claim that for $C$ large enough,  $\Gamma_{k}^v\cap \widetilde \Lambda_{u} \subset (\Gamma_{k}^v\cap \Lambda_{u}) \cup 0_{T^n\times {\mathbb S}^n}$. Indeed, $\Lambda_{u}$ and $\widetilde \Lambda_{u}$ coincide in  $\{ (x,y,X,Y) \mid \vert y \vert \leq C \}$, but outside this set, $\Gamma_{k}^v$ coincides with the zero section. There are actually two types of points in $(\widetilde \Lambda_{u}- \Lambda_{u})\cap \Gamma_{k}^v$ the ones with action $0$, the other with action $A_{u,C}=f_{u,C}(y_{u,C})$ where $f'_{u,C}(y_{u,C})=0$ and $y_{u,C}$ is a non-trivial critical point of $f_{u,C}$. Note that setting $F=f_{u,1}$ we have 
  $f_{u,C}(y)= kCF( \frac{y}{C})$. So,  $f'_{u,C}(y)=kF'( \frac{y}{C})$ and $y_{k,C}=Cz$ where $z$ is a non-trivial critical point of $F$, and $f_{u,C}(y_{u,C})=kCF(z)$. Thus if $f_{u,C}(y_{k,C})\neq 0$ we have that for $C$ large enough, the critical value is outside any given interval. 
 
 Now since $\Gamma_{k}^v \overset{h} \longrightarrow \overline \Gamma^v$, where $\overline \Gamma^v$ is the graph of $\overline\Phi^v$ in the $\gamma$-completion $\widehat{\mathcal L}$, and provided we have $FH^*(\overline \Gamma^v, \widetilde\Lambda_{u}, c- \varepsilon , c+ \varepsilon ) \neq 0$, 
 according to Corollary \ref{Cor-8.2}, this  implies for $k$ large enough
  $FH^*(\Gamma_{k}^v, \widetilde\Lambda_{u}, c- \varepsilon , c+ \varepsilon ) \neq 0$ and as we saw that $FH^*(\Gamma_{k}^v, \widetilde\Lambda_{u}, c- \varepsilon , c+ \varepsilon )=FH^*(\Gamma_{k}^v, \Lambda_{u}, c- \varepsilon , c+ \varepsilon )$ we have a fixed point with action in $[c- \varepsilon, c+ \varepsilon ]$ (hence average action in $[\frac{c- \varepsilon}{v}  ,  \frac{c+ \varepsilon}{v}]$). 
 Now \begin{gather*} FH^*(\overline \Gamma^v, \widetilde\Lambda_{u}, a,b)=H^*(f_{u,C}(x,y)-v\cdot \overline H(y);a,b)=\\ =H^*(v\cdot \overline H_{u/v};a,b)=H^*(\overline H_{u/v}, \frac{a}{v}, \frac{b}{v})\end{gather*}  where $\overline{H}_u(y)=\langle u,y\rangle -\overline {H}(y)$, hence $H^*(\overline H_{u/v}; c/v-\varepsilon, c/v+ \varepsilon)\neq 0$ is equivalent to the existence of $p$ such that $d_s \overline H(p)=u/v$ and $\overline H_{u/v}(y)=c/v$ according to Appendix \ref{Appendix-crit}, Proposition \ref{Prop.10-5}.
 \end{proof} 
 \begin{proof}[Proof of Theorem \ref{Thm-weak-Conley}] 
Statement (a) is just Proposition \ref{Prop-9.1}. For statement (b), let us assume that $\overline H \not\equiv 0$. Then the set $\{ d_s \overline H (p) \mid p \in {\mathbb R}^n \}$ has non-empty interior according to Lemma \ref{fundlemma} of section \ref{Appendix-crit}. There are thus infinitely many rational, non-collinear values of $\alpha$ such that we have a periodic orbit of rotation vector $\alpha$.

Let us now consider the case where $\overline H\equiv0$.  This means in particular that  $\lim_{k}\frac{1}{k} c_{\pm}(\varphi^k)=0$ (we refer to \cite{Viterbo-STAGGF} for the definition of the capacities $c_\pm$). 
But since $c_{+}(\varphi)=c_{-}(\varphi)=0$ if and only if $\varphi=\Id$,  there is an infinite sequence of $k$ such that either $c_{+}(\varphi^k)>0$ or $c_{-}(\varphi^k)<0$. 
Replacing $\varphi$ by $\varphi^{-1}$ we may always assume we are in the first case. Then the fixed point $x_{k}$ of $\varphi^k$ corresponding to $c_{+}(\varphi^k)$ is such that its action, $A(x_{k}, \varphi^k)=c_{+}(\varphi^k)$. 
In case $x_{k}$ is the fixed point of $\varphi$, we get that $A(x_{k},\varphi^k)=k\cdot A( x_{k}, \varphi)$. More generally if $x=x_{pj}=x_{pk}$ we get $\frac{1}{pj}A(x, \varphi^{pj})=\frac{1}{pk}A(x, \varphi^{pk})$, so $\frac{1}{pj}c_{+}(\varphi^{pj})=\frac{1}{pk}c_{+}( \varphi^{pk})$, but since the sequence $\frac{1}{k} c_{+}(\varphi^k)$ is positive and converges to zero, it is non constant and takes infinitely many values. Thus, there are infinitely many fixed points.

More precisely, let us assume $\frac{1}{k} c_+(\varphi^k)=\frac{1}{l} c_+(\varphi^l)$, and $x_k, x_l$ are the corresponding fixed points of $\varphi^k, \varphi^l$. Then if $x_k=x_l$ we must have $\varphi^d(x_k)=x_k$ where $d$ is the $lcd(k,l)$. In particular if $A_1$ is the set of non-zero actions for the fixed points of $\varphi$, and we assume this set is finite, $gcd(k,l)=1$ implies $\frac{1}{k} c_+(\varphi^k)=\frac{1}{l} c_+(\varphi^l)\in A_1$. But the finiteness of $A_1$ and the fact that $\frac{1}{k} c_+(\varphi^k)$ goes to zero with $k$ implies that this is impossible for $k$ large enough. Thus taking the sequence $x_{p_j}$ 
where $p_j$ is the $j$-th prime number. The prime number theorem of Hadamard and La Vall\'ee Poussin then implies that there are asymptotically at least $\frac{N}{\log (N)}$
 distinct periodic orbits of period less than $N$.  This proves the second estimate. The first follows from the formula $\sum_{p\leq N}p \simeq \frac{N^2}{2 \log(N)}$.

\end{proof}

\section{Proof of Lemma \ref{Lemma-3.4} and  the ergodicity of the invariant measures}
We consider an invariant measure of the form $\chi(q,p)\omega^n$.
The  rotation vector is defined, at least when either $X_H$ or $\meas$ are compact supported, and $\meas$ is an invariant measure for $X_H$,  as
 \begin{align*} r(\meas):&  H^1(M, {\mathbb R} ) \longrightarrow {\mathbb R} \\ & \tau    \longrightarrow \int_{S^1\times M}  i_{X_H}\tau   d\meas   \end{align*} 
 
 \begin{proof}[Proof of Lemma \ref{Lemma-3.4}]
 We prove that the rotation number and mean action of $\chi(q,p)\omega^n$ are $0$ and $-(n+1)\frac{\mathrm {Cal}(\varphi)}{\mathrm{vol}(\mathrm{supp}(\varphi))}$ respectively. Since $H$ vanishes unless $\chi=1$, we have $r(\chi\omega^n)=r(\omega^n)$ and is given by: 
 \begin{gather*} \int_{S^1\times M} (i_{X_H}\tau) \omega^n \wedge dt= \int_{S^1\times M}  i_{X_H}\tau \wedge d\lambda \wedge  \omega^{n-1} \wedge dt=-\int_{S^1\times M} d(i_{X_H}\tau)\wedge \lambda \wedge \omega^{n-1}\wedge dt=\\ \int_{S^1\times M} (i_{X_H}d\tau - L_{X_H}\tau) \wedge \lambda \wedge \omega^{n-1}\wedge dt = - \int_{S^1\times M}  L_{X_H}\tau \wedge \lambda \wedge \omega^{n-1} \wedge dt=\\ \int_{S^1\times M}  \tau \wedge (L_{X_H}\lambda )\wedge \omega^{n-1}\wedge dt =  \int_{S^1\times M}  \tau \wedge d(i_{X_H}\lambda + H)\wedge \omega^{n-1}\wedge dt \end{gather*} 
 But this last is the integral of the differential of a (compact supported) exact form, $(i_{X_H}\lambda+H)\tau \wedge \omega^{n-1} \wedge dt$, hence it vanishes. 
 
Moreover the average action of this measure is given by $$\int_{S^1\times M} \left (i_{X_H}\lambda -H\right ) \omega^n\wedge dt$$ since $H$ vanishes outside the support of $\varphi$.  
To compute this integral, we first notice that since $\lambda\wedge \omega^n=0$, we have
\begin{gather*} 0=i_{X_H}(\lambda\wedge \omega^n)=(i_{X_H}\lambda ) \omega^n -n\lambda \wedge i_{X_H}\omega \wedge \omega^{n-1}= (i_{X_H}\lambda ) \omega^n +n\lambda \wedge dH \wedge \omega^{n-1}
\end{gather*} 

so by Stokes's formula $$\int_{S^1\times M} (i_{X_H}\lambda ) \omega^n=-n \int_{S^1\times M}\lambda \wedge dH \wedge \omega^{n-1}=-n \int_{S^1\times M} H \wedge \omega^{n}$$
\begin{gather*} 
\int_{S^1\times M} \left (i_{X_H}\lambda -H\right ) \omega^n\wedge dt =-(n+1)\int_{S^1\times M} H\wedge \omega^n\wedge dt =-(n+1) \Cal(\varphi)
\end{gather*} 

where $\Cal (\varphi)$ is the Calabi invariant of $\varphi$. 
\end{proof} 

Now let us consider the subsets in ${\mathbb R}^{n+1}$ given by  $$\overline R(H)=  \{(\alpha,\langle p,\alpha\rangle -\overline H(p)) \mid \alpha \in \partial_C \overline H(p)\} $$
and 
$$R(H)=  \{(\alpha,A) \mid \exists \meas , \rho(\meas)=\alpha, \varphi_H^*(\meas)=\meas,  A_H(\meas)=A \} $$

We proved in the previous sections that $\overline R (H) \subset R (H)$. Moreover for each element $(\alpha, A)$ in $\overline R (H)$ there is a well defined measure $\meas_{\alpha , A}$ such that $ \rho(\meas_{\alpha , A})=\alpha, \varphi_H^*(\meas_{\alpha , A})=\meas_{\alpha , A},  A_H(\meas_{\alpha , A})=A$. 

We want to figure out whether the measures thus found are ergodic, or whether we can find the minimal number of ergodic measures. Indeed, we assume there are ergodic measures $\meas_1, ..., \meas_q$ generating all the measures we obtained. For this, we need the measures we found to be contained in a polytope with $q$ vertices. Since we know that if $\overline H\not\equiv 0$ the projection of  $\overline R (H)$ on $ {\mathbb R}^{n}$ contains an open set (Lemma \ref{fundlemma} of section \ref{Appendix-crit}), hence the same holds for $R(H)$ and we must have $q\geq n+1$. If moreover  $R (H)$ contains an open set in ${\mathbb R}^{n+1}$, necessarily we shall have $q\geq n+2$.  
We thus proved

\begin{proposition} 
If $\overline H \not\equiv 0$ then there are at least $n+1$ distinct ergodic measures for $\varphi^1$. 
\end{proposition} 
In fact by our argument,  there is at least one ergodic measure for each extreme point of $\bar R(H)$. 

\section{The case of non-compact supported Hamiltonians}

A priori our results only deal with compact supported Hamiltonians. However the same truncation tricks as in \cite{STH} allows one to extend homogenization to coercive Hamiltonians (see section 12 of \cite{STH}) and to prove the following statements whose proofs are  left to  the reader 

\begin{theorem} \label{main-theorem-coercive}
Let $H(q,p)$ be a coercive Hamiltonian in $T^{*}T^{n}$, that is 
$$\lim_{ \vert p \vert \to +\infty} H(q,p)=+\infty$$
and denote by $\overline H (p)$ its homogenization defined in \cite{STH}. 
Let   $\alpha \in \partial_C \overline H (p)$
Then there exists, for $k$ large enough, a solution  of $\phi^{k}(q_{k},p_{k})=(q_{k}+ k \alpha_{k} ,p'_{k})$ 
(with $\lim_{k}\alpha_{k}= \alpha$) and average action 
$$A_k = \frac{1}{k} \int_{0}^k [\gamma_{k}^*\lambda -H(t,\gamma_{k}(t)) ]dt$$ where 
$\gamma_{k} (t)= \varphi^{t}(q_{k},p_{k})$. Moreover as $k$ goes to infinity $A_{k}$ converges
  to $$\lim_{k}A_{k}=p\cdot \alpha-\overline H(p)$$ Therefore there exists an invariant 
  measure $\meas_{\alpha}$ with rotation vector $\alpha$ and average action  
  $${\mathcal A}(\meas_{\alpha})\overset{def}=\int_{T^*T^n} [p\frac{\partial H}{\partial p}(q,p)-H(q,p) ]d\meas_{\alpha}= p\cdot \alpha-\overline H(p).$$
\end{theorem} 

In particular for $H(q,p)$ strictly convex in $p$ and superlinear, we have that $\overline H(p)$ is also convex in $p$ (see proposition 12.4 in \cite{STH}) and superlinear,  so for each $\alpha$ there exists  a unique convex set $C_\alpha$ in $ ({\mathbb R}^n)^*$ such that 
$p \in C_\alpha$ if and only if $\alpha \in \partial_C\overline H(p)$. Note that where $\overline H$ is strictly convex, $C_\alpha$ is reduced to a point, and that in this case Mather's theory is much more precise, and tells us that 
the measure obtained are minimal, and its support is the graph of a Lipschitz function over a subset of $T^n$.  

The above theorem extends as in section 11 of \cite{STH} to the coercive time-dependent (i.e. non-autonomous) case, that is $H(t,q,p)$ with $H(t,q,p)$ $T$-periodic in $t$. We then again have a homogenized Hamiltonian $\overline H$. 
The continuity of H is not known in general, though H is always lower semi-continuous. Continuity does hold in certain special cases: e.g., for H convex or globally Lipschitz at infinity.

\section{Appendix 1: Critical point theory for non-smooth functions and subdifferentials}\label{Appendix-crit}

The aim of this section is to clarify the notions of differential that occur crucially in the previous sections. Indeed, restricting the set of rotation vectors of invariant measures to the values corresponding to regular points is not an option, since even in the convex case, the function $\overline H$ has generally dense subsets of non-differentiable values. We shall deal with two situations. The first one corresponds to Lipschitz functions: these occur as homogenization of $C^1$ (or Lipschitz) 
Hamiltonians, which are the only ones we encounter in practice. This is the subject of the first subsection, and uses analytic tools, basically a notion of subdifferential and a suitable version of the Morse deformation lemma. The second one applies to any continuous function. It is best suited to our general line of work, and in principle allows us to use the main theorem in the case of Hamiltonians belonging to the Humili\`ere completion, even though one should formalize the notion of invariant measure for such objects. 

\subsection{Analytical theory in the Lipschitz case}\label{sec:10}
While the critical point theory has been studied for (smooth and non-smooth) functionals on infinite dimensional spaces, we shall here restrict ourselves to the finite dimensional case. 
Remember that for a sequence of sets, $(X_k)_{k\in \mathbb N}$, we set
$$\limsup_k X_k = \bigcap_p \bigcup_{k\geq p} X_k$$
$$\liminf_k X_k = \bigcup_p \bigcap_{k\geq p} X_k$$
and for a function $f$ we defined
$$f^c=\left\{x \in M \mid f(x)<c \right \}$$
First assume $f$ is Lipschitz on a smooth manifold $M$. Then we define 

\begin{definition}\label{Def-12.1} Let $f$ be a Lipschitz function. 
The vector $w\in T_x^*M$ is in $\partial_C f(x)$ the Clarke differential of $f$ at $x$, if and only if $$ \forall v \in E\; \; \limsup_{h \to 0, \lambda\downarrow 0} \frac{1}{\lambda} [f(x+h+\lambda v)-f(x+h) ] \geq  \langle w, v \rangle $$
\end{definition}

The following proposition describes the main properties of $\partial_C f$

\begin{proposition}[\cite{Clarke2, Chang}]
We have the following properties:

\begin{enumerate} [label=\theenumi]
\item $\partial_C f(x)$ is a non-empty convex compact set in $T_{x}^*M$
\item $\partial_C (f+g)(x)\subset \partial_C f(x)+ \partial_C g (x)$
\item $\partial_C (\alpha f) (x) = \alpha \partial_C f(x)$
\item The set-valued  mapping $ x \mapsto \partial_C f(x)$ is upper semi-continuous. The map $x  \mapsto \lambda_f (x)=\min_{w\in \partial_C f(x)} \vert w \vert$ is lower semi-continuous. 
\item Let $\varphi \in C^1([0,1], X)$ then  $h=f\circ \varphi$ is differentiable almost everywhere (according to Rademacher's theorem) and
$$h'(t) \leq \max \{ \langle w, \varphi'(t)\rangle \mid w \in \partial_C f(\varphi(t))\}$$
\end{enumerate}
\end{proposition}

According to \cite{Clarke3} we have
\begin{proposition} 
Let $E$ be a set of full measure in a neighborhood of $x_0$, and $f$ a Lipschitz function differentiable in $E$. 
then $$\partial_Cf(x)= {\rm Conv}\left\{ \lim_i df(x_i) \mid x_i \longrightarrow x, x_i\in E\right \}$$
\end{proposition} 
this implies
We shall also use the following consequence of thm. 3.1 in \cite{Jourani}

\begin{proposition}\label{Prop-12.4} (\cite{Jourani})
Let $(f_j)_{j\in \mathbb N}$ be a sequence of Lipschitz functions with bounded Lipschitz constant,  having $C^0$-limit $f$.
Then $$\partial_C f(x_0) \subset \limsup_{\substack{j\to +\infty\\ x\to x_0}} \partial_C f_j(x)$$
\end{proposition} 
\begin{definition}Let $f$ be a Lipschitz function. 
We define the set of critical points at level $c$ as  $K_c=\{ x \in f^{-1}(c) \mid 0 \in \partial_C f (x) \}$.  We set $\lambda_f(x)=\inf_{w\in \partial_C f(x)} \Vert w \Vert \}$
\end{definition}

\begin{definition}\label{Definition-PS2} Let $f$ be a Lipschitz function. 
We shall say that $f$ satisfies the Palais-Smale condition ((PS) condition for short) if
 any sequence $(x_n)_{n\geq 1}$ such that $(f(x_n))_{n\geq 1}$ converges and $ \lim_n \lambda_f (x_n)=0$, $(x_n)_{n\geq 1}$ has a converging subsequence. 
\end{definition}

The crucial fact is the existence of a pseudo-gradient vector field in the complement of $K_c$. We denote by $N_\delta(K_c)$ a $\delta$-neighborhood of $K_c$. 

\begin{lemma}[Lemma 3.3 in \cite{Chang}]\label{Lemma-10.7}
Assume $f$ satisfies the Palais-Smale condition. There exists a Lipschitz vector field $v(x)$ defined in a neighborhood of $B(c, \varepsilon , \delta)=( f^{c+ \varepsilon}\setminus f^{c-\varepsilon})\setminus N_\delta(K_c)$ such that $ \Vert v(x) \Vert  \leq 1$ and
 $\langle v(x), w \rangle \geq \frac{b}{2}$ for all $w \in \partial f (x)$, where $0<b= \inf\{\lambda_f (x) \mid x \in B(c,\varepsilon, \delta)\}$.
\end{lemma}
Note that the Lipschitz pseudo-gradient can be regularized to a smooth one if one wishes, since away from the critical points, being a pseudo-gradient is an open condition. 
From this we see that if $\overline U\cap ( f^{c+ \varepsilon}\setminus f^{c-\varepsilon})$  does not intersect the set of critical points, we may  follow the flow of the vector field $v$, so that if $C$ is a cycle representing a homology class in $H_{*}(U\cap f^{c+ \varepsilon }, U\cap f^{c- \varepsilon })$ for $ \varepsilon $ small enough, then the flow of $v$ applied to $C$ shows that $C$ is homologous to a cycle in $U\cap f^{c- \varepsilon }$, hence $C$ is homologous to zero. This brings us to the following subsection.

\subsection{Topological  theory (according to  \texorpdfstring{\cite{Vichery2}}{Vichery2})}\label{topt}

Let $f$ be a continuous function on $X$. We define a {\it strict critical point} of $f$, as follows

\begin{definition}\label{def-10.6} Let $f$ be a continuous function. 
We define the set of {\it strict critical points} at level $c$ as  the set of points such that $U$ denoting a neighborhood of $x$, 
$$\lim_{U\ni x}\lim_{\varepsilon \to 0} H^*(U\cap f^{c+ \varepsilon}, U\cap f^{c}) \neq 0$$
If $f^{-1}(c)$ contains a strict critical point, it is called a critical level. Other points are called weakly regular points. 
\end{definition}

Note that we have for $\eta < \varepsilon $  maps
$$H^*(U\cap f^{c+ \varepsilon }, U\cap f^{c}) \longrightarrow H^*(U\cap f^{c+ \eta }, U\cap f^{c})$$
 while for $V \supset U$  we have a map
$$H^*(V\cap f^{c+ \varepsilon }, V\cap f^{c}) \longrightarrow H^*(U\cap f^{c+ \varepsilon  }, U\cap f^{c })$$
and the  limits  in the definition should be understood as direct limits. 

Even if $f$ is smooth, this does not coincide exactly with the usual notion of critical and regular point. For example if $f(x)=x^3$, the origin is critical but weakly regular, since there is no topological change for the sublevels of $f$ at $0$. 

From  Lemma \ref{Lemma-10.7} the first part of the following proposition follows 

\begin{proposition}Let $f$ be Lipschitz and satisfying the Palais-Smale condition above. Then   strict critical points at level $c$ are contained in 
$K_c$. Moreover if $f$ has a local maximum (resp. minimum) at $x$, then $x$ is a strict critical point. 
\end{proposition}

\begin{proof} 
The second statement follows obviously from the fact that for a local minimum, we have  $H_0(f^{c+ \varepsilon }\cap U, f^{c- \varepsilon }\cap U)=H_0(f^{c+ \varepsilon }\cap U, \emptyset)\neq 0$ since $f^{c+ \varepsilon}\cap U$ is non-empty while $f^{c- \varepsilon}\cap U$ is empty. 

\end{proof}

\begin{definition} 
We denote by $d_tf(x)$ the set of $p$ such that there exists a smooth function $\varphi$ with $d\varphi(x)=p$ and $f(x)-\varphi(x)$ has a strict critical point at $x$, where the scalar product is the euclidean one in some chart near $z$ (so $f_p$ is only defined near $z$). This is called the topological differential at $z$. The set of all limits of  $d_{t}f(x_n)$  as $x_n$ converges to $z$ is denoted by $D_{t}f(z)$.
\end{definition} 

\begin{remark} 
The set $D_tf(x)$ coincides with $\partial f(x)$ as defined in Definition 3.6 of \cite{Vichery2}. 
 \end{remark} 
\begin{proposition}
The set $D_tf(z)$ is contained in $\partial_C f(z)$ and the convex hull of $D_tf(z)$  equals $\partial_C f(z)$. In particular if $f$ is $C^1$, $D_tf(x)=\{df(x)\}$. 
\end{proposition}
\begin{proof}
This is theorem 3.14 and 3.20 of \cite{Vichery2}. 
\end{proof}
The above notion is analogous to the one defined using microlocal theory of sheaves of  \cite{Kashiwara-Schapira}, as is explained in \cite{Vichery2}. Indeed, the
 singular support of a sheaf is a classical notion in sheaf theory (see \cite{Kashiwara-Schapira}), defined as follows. The derived category of bounded complexes of sheaves over $N$, $D^b(N)$ is defined as the set of equivalence classes for finite complexes of sheaves
  $$0 \longrightarrow \F_1 \longrightarrow \F_2 \longrightarrow ... .. \longrightarrow \F_k\longrightarrow 0$$ 
 for the relation generated by quasi-isomorphism : $\cF$ and $\cG$ are quasi-isomorphic if there is a morphism between them inducing a cohomology isomorphism. 
 Then 
 $\Gamma (U; \bullet )$ denotes the section functor: $\Gamma (U; \F) $ is the set of sections of $\F$ over $U$. We use the same notation for the extension of $\Gamma (U, \bullet )$ to the set of chain complexes of sheaves: it sends a  complex 
 $$0 \longrightarrow \F_1 \longrightarrow \F_2 \longrightarrow ... .. \longrightarrow \F_k\longrightarrow 0$$ to the complex
 $$0 \longrightarrow \Gamma(U,\F_1) \longrightarrow\Gamma(U,\F_1) \longrightarrow ... .. \longrightarrow \Gamma(U,\F_k)\longrightarrow 0$$
  and $R\Gamma (U; \bullet )$ is the derived functor : we replace the complex $\cF$ by a quasi-isomorphic complex of injective sheaves, $\cI$, then $R\Gamma (U; \bullet )$ is the equivalence class of $R\Gamma (U; \cI )$.
 
  \begin{definition} 
  Let $\cF$ be a sheaf on $X$. Then $(x_{0},p_{0}) \notin SS(F)$ if for any $p$ close to $p_{0}$, and $\psi$ such that $p=d\psi(x)$ and $\psi(x)=0$ we have $$ R\Gamma (\{\psi\leq 0\}, \cF)_{x}=0$$
  This is equivalent to $\lim_{W\ni x}H^*(W, W\cap \{\psi\leq 0\};\cF)=0$. 
  \end{definition} 
  The connection between the two definitions is as follows.  Consider the sheaf $\cF_{f}$ on $M\times {\mathbb R} $ that is the constant sheaf on $\{(x,t) \mid f(x) \geq  t \}$ and vanishes elsewhere. Then $SS(\cF_{f})=\{(x,t,p,\tau ) \mid \tau D_{t}f(x)=p \}$. It is not hard to see  that as expected, $SS(\cF_{f})$ is a conical coisotropic submanifold.

It follows from the   sheaf theoretic Morse lemma from \cite{Kashiwara-Schapira} (Corollary 5.4.19, page 239) that 
\begin{proposition} 
Let $f$ be a continuous function. Let us assume $c$ is a weakly regular level. Then for $ \varepsilon $ small enough,  $H^*(f^{c+ \varepsilon }, f^{c- \varepsilon })=0$.
\end{proposition} 

  \begin{proof}
Let $k_X$ be the sheaf of locally constant functions. If $f^{-1}(c)$ contains no point such that $D_tf(x)=0$, then according to the sheaf-theoretic Morse lemma, 
 $$R\Gamma (f^{c+ \varepsilon}; k_X) \longrightarrow 
  R\Gamma (f^{c- \varepsilon}; k_X)$$ is an isomorphism, but this implies by the long exact sequence in cohomology that
  $H^*(f^{c+ \varepsilon},f^{c- \varepsilon})=0$. 
  \end{proof}
  
  Finally we have
  
  \begin{proposition} \label{local-global-crit}
Let $f$ be a Lipschitz function satisfying the Palais-Smale condition above. Let us assume $f^{-1}(c)$ contains an {\bf isolated} strict critical point. Then for $ \varepsilon $ small enough,  $H^*(f^{c+ \varepsilon }, f^{c- \varepsilon })\neq 0$.
\end{proposition} 

  \begin{proof}
This follows from the fact that if a sheaf is equal to a skyscraper sheaf\footnote{i.e. a sheaf supported at a single point.} near $U$ it has non-trivial sections.  A more elementary approach is as follows. First notice that if we have two sets $B\subset A$ and open sets $\overline U \subset V$ and $A\cap (\overline {(V\setminus U)}=B \cap (\overline {(V\setminus U)}$ then $$H^*(A,B)=H^*(A\cap U, B\cap U) \oplus H^*(A\cap (X\setminus V), B\cap (X\setminus V))$$ Now if $x$ is an isolated strict critical point, of $f$ according to lemma \ref{Lemma-10.7} (i.e. lemma 3.3 in  \cite{Chang}), we can deform $f^{c+ \varepsilon } $ to $f^{c- \varepsilon }$ in $V\setminus U$, for some $x\in U \subset \overline U \subset V$. We assume here that $V$ does not contain other critical points than $x$. Thus $$H^*(f^{c+ \varepsilon }, f^{c- \varepsilon })= H^*(f^{c+ \varepsilon }\cap U, f^{c- \varepsilon }\cap U) \oplus H^*(f^{c+ \varepsilon }\cap (X\setminus V), f^{c- \varepsilon }\cap (X\setminus V)) $$
and since the first term of the right-hand side is non-zero, so is the left-hand side. 

\end{proof}

  Note that if we have an open set $\Omega$ where $f$ is constant, then any $x \in \Omega$ is a strict critical point, but this does not imply $H^*(f^{c+ \varepsilon }, f^{c- \varepsilon })\neq 0$. So the above proposition does not hold if the critical point is not isolated. Take as an example $f(x) <0$ for $x<-1$, $f(x)>0$ for $x>1$ and $f=0$ on $[-1,1]$. Then $H^*(f^b,f^a)=0$ for all $a<b$, while $0 \in d_tf(0)$. 
  
This prompts the following definition

\begin{definition}\label{Def-9.14} The real number $c\in {\mathbb R} $ is a {\it strong  critical value} of $f\in C^{0,1}(X)$ if $\lim_{ \varepsilon \to 0} H^*(f^{c+ \varepsilon }, f^{c- \varepsilon })\neq 0$. 
If $x \in f^{-1}(c)$, $x$ is a {\it strong critical point} if $$\lim_{ \varepsilon \to 0}H^*(f^{c+ \varepsilon }, f^{c-\varepsilon }) \longrightarrow  \lim_{ x\in U}\lim_{ \varepsilon \to 0}H^*(f^{c+ \varepsilon }\cap U, f^{c-\varepsilon}\cap U)$$
is non-zero. 
For $X= {\mathbb R} ^n$, we say that the {\it strong differential} of $f$ at $x_0$ is the set of $\alpha$ such that $f_{\alpha}(x)=f(x)-\langle \alpha, x\rangle$ has a strong critical point at $x_0$. We denote it by $d_sf(x_0)$. Finally we denote by $D_sf(x_0)$ the set of limits of strong differentials at $x_0$, that is the set of limits of $d_sf(x_n)$ such that $x_n$ converges to $x_0$. 
\end{definition}

An application of Mayer-Vietoris implies
\begin{proposition} \label{Prop.10-5}  Let $f$ be a Lipschitz function satisfying the (PS) condition. Assume   $\lim_{\varepsilon \to 0}H^*(f^{c+ \varepsilon}, f^{c- \varepsilon})\neq 0$. Then $f^{-1}(c)$ contains a strong critical point. 
\end{proposition} 
\begin{proof} Indeed, assume for $U=U_j$ small enough the map $H^*(f^{c+ \varepsilon },f^c) \longrightarrow H^*(f^{c+ \varepsilon }\cap U, f^c\cap U)$  vanishes. For typographical reasons, we shall write $H_f^*(A)$ instead of $H^*(f^{c+ \varepsilon }\cap A ,f^c\cap A)$.
We then write the long exact sequence for $U_1\cap U_2$ that is

\xymatrix {
H_f^*(U_1\cup U_2) \ar[r]& H_f^*(U_1) \oplus H_f^*(U_2) \ar[r] & H_f^*(U_1\cap U_2) \ar[r]^{\delta^*} &H_f^{*+1}(U_1\cup U_2) \ar[r]&...\\
H_f^*(X) \ar[r]\ar[u]& H_f^*(X)  \oplus H_f^*(X) \ar[u]^0\ar[r] & H_f^*(X)\ar[u]^0 \ar[r]^0 &H_f^{*+1}(X) \ar[r]\ar[u]&..
}
So we conclude by the $5$-lemma that the map $H_f^*(X) \longrightarrow H_f^*(U_1\cup U_2)$ vanishes. By induction, we get that the map 
$H_f^*(X) \longrightarrow H_f^*(U_1\cup...\cup U_k)$ vanishes. Now the (PS) condition implies that outside a compact set, we may deform $f^{c+ \varepsilon }$ to $f^{c- \varepsilon }$. So by compactness, we may reduce ourselves to the case where $f^{c+ \varepsilon }\setminus f^{c- \varepsilon }$ is covered by a finite number of such open sets, and we eventually get that the identity map $H^*( f^{c+\varepsilon }, f^{c- \varepsilon }) \longrightarrow H^*( f^{c+\varepsilon }, f^{c- \varepsilon })$ vanishes, so $H^*( f^{c+\varepsilon }, f^{c- \varepsilon })=0$.

\end{proof} 
Clearly a strong critical point is a strict critical point, but the converse need not be true. Note that the notion of strong critical point is not purely local.  However the converse holds if either the critical point is isolated, or the critical point is a strict  local minimum (or a strict local maximum), that is there is neighborhood $U$ of $K_c$ the set of critical points at level $c$,  such that for all $y\in U\setminus K_c$ we have $f(y)>f(x)$ (note that this does not imply that the critical point is isolated) 

We now prove that for a smooth function,  the various differentials coincide.
\begin{Cor} \label{Cor-10.16}
For a smooth (i.e. $C^{\infty}$) function $f$ 
and any $x_0\in M$, 
we have 
$$\{ D_sf(x_0) \mid x_0\in f^{-1}(c)\}=\{ df(x_0)\mid x_0\in f^{-1}(c)\}$$
\end{Cor} 
\begin{proof} 
It is clear that for a smooth function, $f$, if $d_sf(x)$ is non-empty, it is equal to the singleton containing $df(x)$. Now it is enough to show that if $df(x_0)=0$, $f(x_0)=c$, there is a sequence $x_n$ such that $df(x_n)=\alpha_n$,  $x_n$ is an isolated solution of $df(x)=\alpha_n$ and $\lim_n f(x_n)=c$. But by Morse-Sard's theorem, the set of values of $df$ at which $d^2f(x)$ is degenerate has measure zero, so we can find a sequence $\alpha_n \to 0$ such that $f(x) - \langle \alpha_n,x\rangle$ is Morse, hence $\alpha_n=d_sf(x_n)=df(x_n)$, $df(x_n) \to df(x_0)$ and $\lim_n f(x_n)=f(x_0)=c$.  \end{proof} 

\begin{remark} 
For example for $f(x)=x^3$, we see that $d_sf(0)=\emptyset$, but    for $x_0\neq 0$ $d_sf(x_0)=3x_0^2$ because $x^3-3x_0^2x$ has a strict local minimum at $x_0$ since $(x^3-3x_0^2x)''=3x_0^2 >0$. However $\lim_{x\to 0} d_sf(x_0)=0=D_sf(0)$. 

One thus gets the impression that $D_sf(x_0)=\{df(x_0)\}$ for all $x_0$ but this does not hold in general. Take as a counterexample a function $f$ such that $f(x)<0$ for $x<-1$, $f(x)>0$ for $x>1$, in both regions $f'(x)>0$ and finally  $f=0$ on $[-1,1]$.

\end{remark} 

\begin{proposition} \label{near-strong}
Assume $(f_k)_{k\geq 1}$ is a sequence of continuous functions that $C^0$-converges to $f$. Then if $p\in d_sf(x)$ there is a sequence $(x_k)_{k\geq 1}$ such that for $k$ large enough, $p\in d_sf_k(x_k)$. In particular, if $\alpha$ belongs to the set $\{d_sf(x)\mid x \in N\}$, then, for $k$ large enough, $\alpha \in \{d_sf_k(x)\mid x \in N\}$
\end{proposition} 

\begin{proof} 
Indeed, this follows from the fact that $$ H^*(g^{c+ \varepsilon }, g^{c- \varepsilon })= \lim_k H^*(g_k^{c+ \varepsilon }, g_k^{c- \varepsilon })$$ applied to $g=f_p$,  so if $H^*(f_p^{c+ \varepsilon }, f_p^{c- \varepsilon })\neq 0$ the same holds for $(f_k)_p$ for $k$ large enough. 
\end{proof} 

\begin{Cor}
    Let $h=\lim_j h_j$ where the $h_j$ are smooth in the complement on an closed set of measure zero (depending on $j$). Then
    $$\partial_Ch(x)\subset \mathrm{Conv} \left\{ \lim d_sh_j(x_j) \mid \lim_j x_j=x\right \}$$
\end{Cor}
\begin{proof}
    We have by Proposition A3 and Corollary A.19 that for all $j$ $$dh_j(x)\in \left\{\lim d_sh_j(x_k)\mid \lim h_j(x_k)=h_j(x)\right\}$$ for $x$ in the smooth open region. Since by Proposition A.5 $$\partial_Ch_j(x)= \mathrm{Conv} \left\{lim_{x_k\to x} dh_j(x_k)\right\}$$ 
    we get
   \begin{gather*}
       \partial_Ch(x) \subset \lim_j \partial_Ch_j(x)\subset \mathrm{Conv} \left\{ \lim_{j,k} d_sh_j(x_k) \mid h_j(x_k) \to h(x)\right \}=\\
    \mathrm{Conv} \left\{ \lim d_sh_j(x_j) \mid \lim_j h(x_j)=h(x)\right \}
\end{gather*}
\end{proof}
\subsection{A lemma on the set of subdifferentials}\label{Section-9}

 We  have
\begin{lemma}\label{fundlemma} Let $f$ be a compact supported function on $ {\mathbb R}^n$. If $f$ is non-constant then the set of $d_s f (x)$ as $x$ describes $ {\mathbb R}^n$ must contain a neighborhood of $0$. More precisely if  $\supp(f)\subset B(0,1)$, we have $$\{ d_sf(x) \mid x \in B(0,1)\} \supset B(0, \Vert f\Vert_{C^0}/4)$$
\end{lemma} \label{lemma-8.18}
\begin{proof} Assume for simplicity that $f$ vanishes outside the unit ball, $B$. 
Consider the function $f_{p}(x)= f(x)- \langle p , x-z \rangle$. We claim that for $p$ small enough this function has either a local  minimum or a local maximum and therefore $p \in \partial f(z)$. Indeed, $f=f_{0}$ has either a strictly negative minimum or a strictly positive maximum. Assume we are in the first case and let $x_0$ be the minimum. Then $f(x_0)\leq - \varepsilon_0\leq \min_{u \in \partial B} f(u)- \varepsilon_{0}$ for some   $ \varepsilon _{0}>0$. For $p$ small enough (take $ \vert p \vert \leq \frac{ \varepsilon_0}{4}$), the same holds for $f_{p}$ with a smaller constant, that is $$f_{p}(x_{0})\leq \min_{u \in \partial B} f_{p}(u)- \frac{1}{2} \varepsilon_{0}$$ As a result $f_{p}$ must have a global minimum, which is necessarily a strong critical point,  so that $0 \in \{ d_sf_p(x) \mid x \in B\}$,
 in other words $p\in\{ d_sf(x) \mid x \in B\}$. 
\end{proof} 

\subsection{Subdifferential of  selectors}

We shall need some  classical results about the connection between the different subdifferentials. First of all if $f$ is a bounded function defined on a Banach space, we define
the approximate differential of $f$ at $x$, denoted by $\partial_Af(z)$ as follows (we refer to \cite{Ioffe1} for details).
First
$$d^-f(x,h)= \liminf_{\substack{u\to h \\ t\searrow 0}} t^{-1}\left ( f(x+tu)-f(x) \right )  $$
then
$$\partial^-f(x)=\{ x^* \in X^* \mid \langle x^*, h\rangle \leq d^-f(x,h), \forall h\}$$
and finally the approximate differential is defined as 
$$\partial_Af(z)=\limsup_{\substack{ x\to z\\f(x) \to f(z)}} \partial^-f(x)$$
Note that in our case $f$ will always be continuous, so the condition $f(x) \to f(z)$ is not necessary in the definition. An immediate consequence of the definition (see \cite{Ioffe1} (1.1), p. 392) is that
$$\partial_Af(z)=\limsup_{x\to z} \partial_Af(x)$$

Let $(L_k)_{k\geq 1}$ be a sequence of smooth Lagrangians Hamiltonianly isotopic to the zero section in $T^*(N\times M)$ such that $L_k$ $\gamma$-converges to $L\in \widehat {\mathcal L}$. 
Let $u_k(x)=c(\theta, L_{k,x})$ and $u(x)=c(\theta, L_x)$ where $\theta \in H^*(M)$ and $x\in N$. 
\begin{definition} \label{Definition-10.19}
Let $L$ be an exact Lagrangian  in $T^*M$ and $f_L$ such that $\lambda_{L}=df_L$. We denote by ${\rm{Conv}}_x(L)$  the union of the convex envelopes of the $L\cap T_x^*M$. We denote by ${{\rm{LConv}}_x}(L)$  the union of the convex envelopes of the $L\cap T_x^*M \cap f_L^{-1}(c)$ where $c$ varies in $ {\mathbb R}$.

More generally if $L$ is Lagrangian in $T^*(N\times M)$, we denote by ${\rm{Conv}}_{N,x}(L)$ the union of the convex hulls of $L\cap (0_N\times T_x^*M)$, and 
${{\rm{LConv}}_{N,x}}(L)$ the union of the convex hulls of $L\cap (0_N\times T_x^*M)\cap f_L^{-1}(c)$ as $c$ varies in $ {\mathbb R}$ .
\end{definition} 
Note that   ${\rm{LConv}}_{N,x}(L) \subset {\rm Conv}_{N,x}(L)$ and\footnote{${\rm{LConv}}$ is for Levelwise Convex hull.} both are closed sets. 

\begin{figure}[H]
\begin{overpic}[width=6cm]{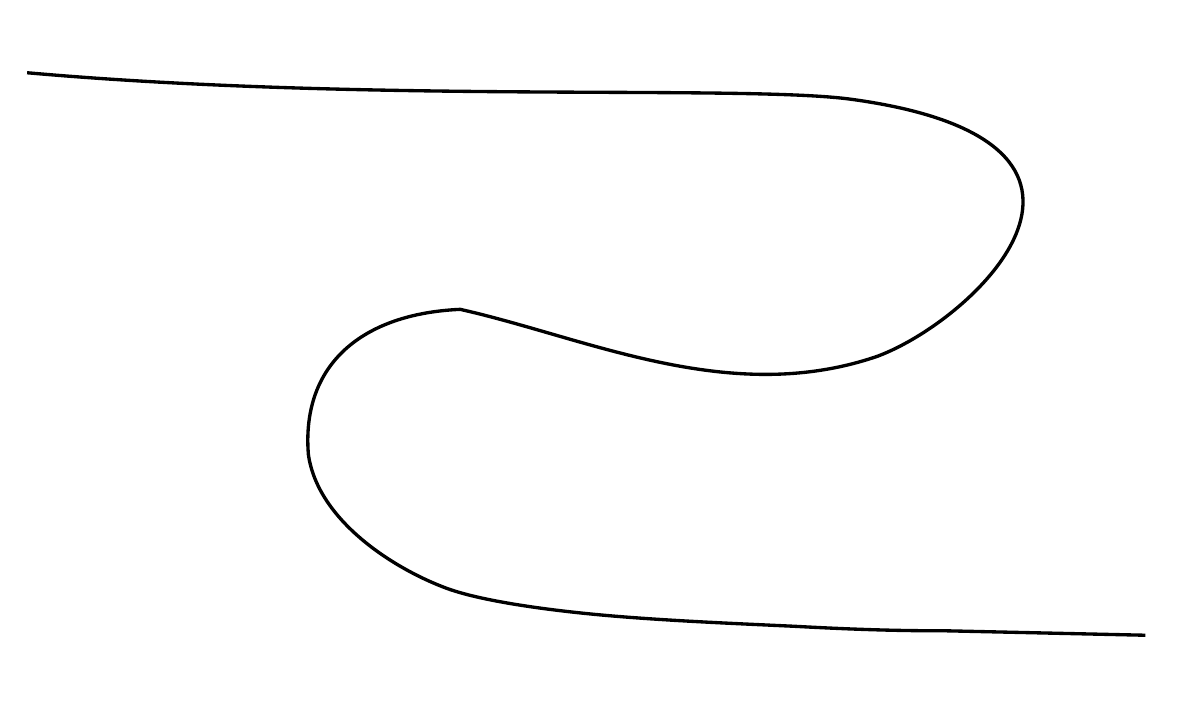}
    \put (95,12){$L$}  
\end{overpic}
\begin{overpic}[width=6cm]{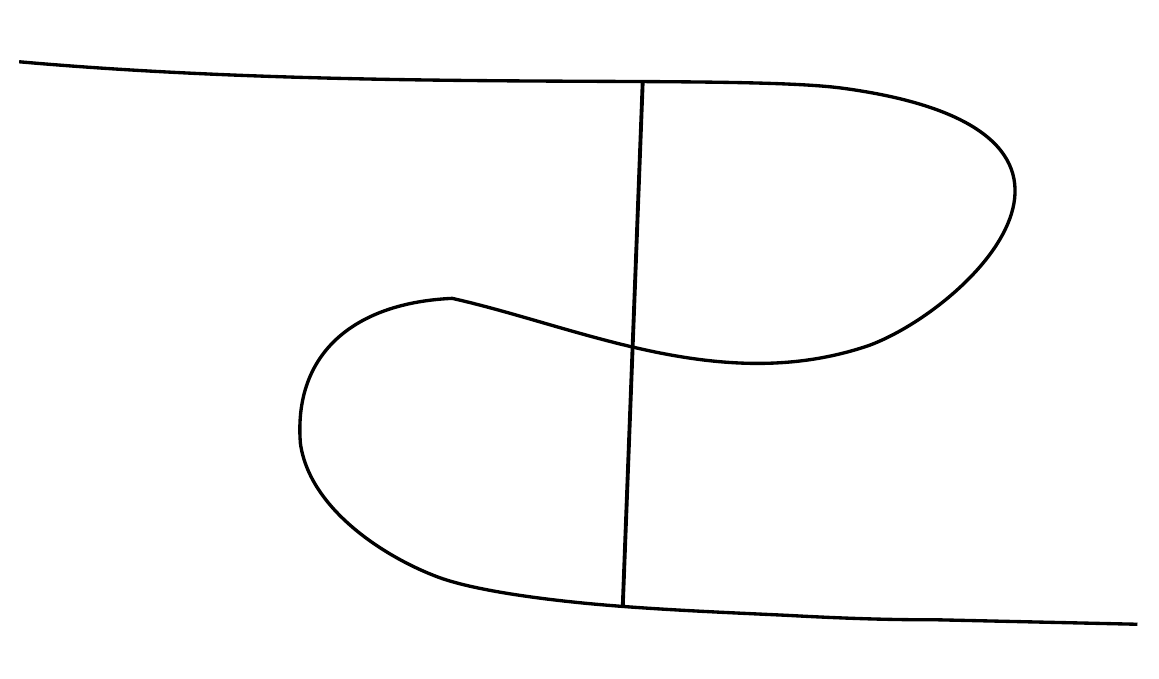}
    \put (95,12){${{\rm{LConv}}_{N}}(L) $}  
\end{overpic}
\center\begin{overpic}[width=6cm]{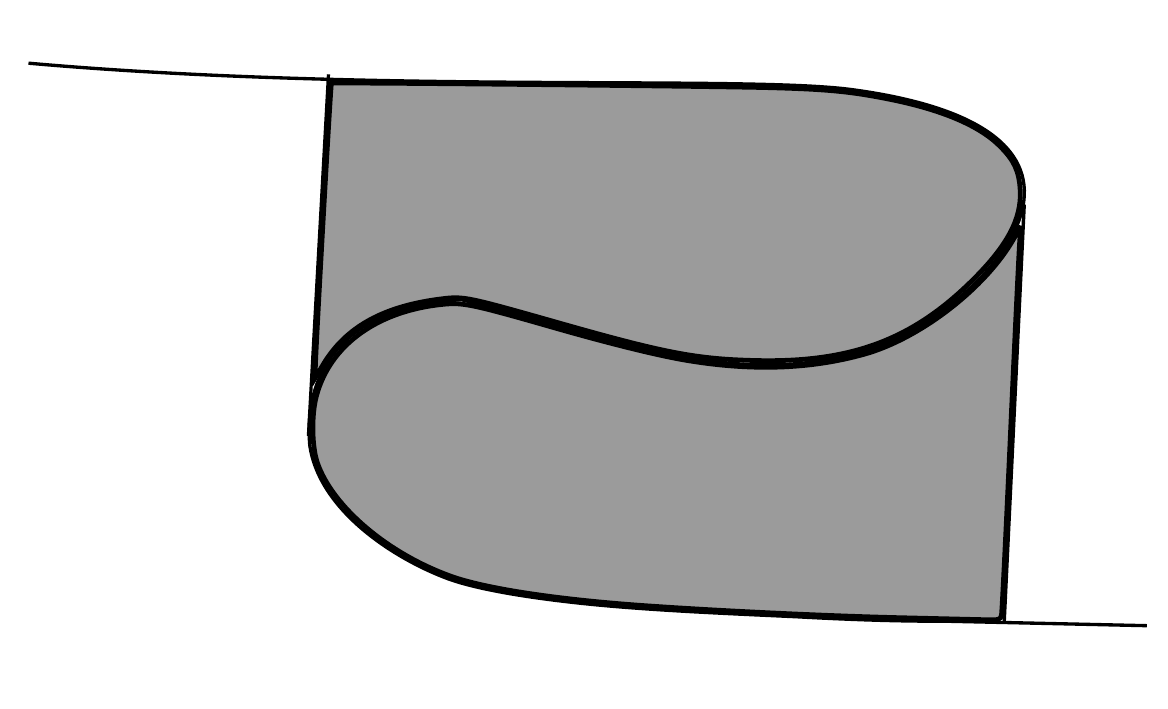}
    \put (95,12){${\rm Conv}_{N}(L)$}  
\end{overpic}
\caption{The Lagrangian $L$ and the sets ${{\rm{LConv}}_{N}}(L) $ and ${\rm Conv}_{N}(L)$. For ${{\rm{LConv}}_{N}}(L) $ the two "ears" have equal area.
}\label{fig-1}
\end{figure}

The following can be considered as an extension of theorem 2.1 (4) page 251 of \cite{Clarke1} (for the case where $\theta$ is the fundamental class of $M$). 

\begin{proposition} \label{Prop-10.17}
Let $(L_k)_{k\geq 1}$ be a sequence of smooth Lagrangians Hamiltonianly isotopic to the zero section in $T^*(N\times M)$ such that $L_k$ $\gamma$-converges to $L\in \widehat {\mathcal L}$.
We have \begin{enumerate} [label=\theenumi] 
\item for all $k$, $\partial_C u_k(x) \subset {\rm{LConv}}_{N,x}(L_k)$ 
\item  $\partial_Cu(x) \subset {\rm Conv}_{N,x}(\limsup_k  {\rm{LConv}}_{N,x}(L_k))$. 
\end{enumerate} \end{proposition} 
\begin{proof} 

The second result follows from the first part and from the fact that according to \cite{STH}, $u_k$ $C^0$-converges to $u$. It then  follows from \cite{Ioffe1} and \cite{Jourani} (thm 3.2, using the fact that according to thm 2.1 $\partial_Af=\partial_Gf$),   that for any sequence $u_k$ of  Lipschitz functions converging to the Lipschitz function $u$, we have $\partial_Au(x)\subset  \limsup_k \partial_A u_k(x_k)$ where $\partial_A$ is the approximate differential, and since the closed convex hull of $\partial_Au(x)$ is $\partial_Cu(x)$ (see \cite{Ioffe1} thm. 2) we get 
$\partial_C u(x)\subset {\rm Conv}_{N,p}\{ \limsup_k \partial_C u_k(x_k)\}$ where we take all limits such that $x_k$ converges to $x$. 

Let us now prove the first part. We can fix $k$ and write $L$ for $L_k$, $u$ for $u_k$. 
Let $L$ be such that  there exists $\Sigma$ of measure zero such that on $N\setminus \Sigma$ we have a G.F.Q.I. $S(q,x,\xi)$ of $L$ and $S(\bullet ,x,\bullet )$ is Morse and has all distinct critical values, with critical points $q^r(x), \xi^r(x)$ and $1\leq r \leq s$ ($s$ is only constant on each connected component of $N\setminus \Sigma_k$).  This property is generic for the $C^\infty$ topology. Then we have on $N \setminus \Sigma$ that $c(\theta,L_x)=S(q^r(x),x,\xi^r(x))$ where $x\mapsto q^r(x),\xi^r(x)$ for $r \in 1,...p$ is a smooth function on $N\setminus \Sigma$. Now $ \frac{\partial S}{\partial q}(q^r(x),x,\xi^r(x))= \frac{\partial S}{\partial \xi}(q^r(x),x,\xi^r(x))=0$, so that $ \frac{d}{dx}c(\theta\otimes 1(x),L)= \frac{d}{dx} S(q^r(x),x,\xi^r(x))= \frac{\partial S}{\partial x}(q^r(x), x,\xi^r(x)) \in L$ for some $r$. Since $L$ is closed, and according to \cite{Clarke2} ( proposition 5, p. 167), $$\partial_Cu(x)={\rm Conv}_p\{ \lim_l du(x_l) \mid x_l \to x, x_l\in \Omega \}$$ where $\Omega$ is any set of full measure in the set of differentiability points of $u$, we get, using also that $u(x)=S(q^r(x),x,\xi^r(x))$ that  $\partial_Cu(x) \in {\rm{LConv}}_{N,p}(L)$. By $C^\infty$ density, we can always perturb $L$ so that they are generic in the above sense, and if $L_{l} \to L$ converges in the  $C^\infty$ topology as $l$ goes to infinity, we have $\partial_Cu=\lim_l\partial_Cu_{l}$, and then $\lim_l L_{l}=L$ we get 
$$\partial_C u(x) \subset {\rm Conv}_p(\limsup_l \partial_C u_l) \subset {\rm Conv}_p(\limsup_l {\rm{LConv}}_p(L_l))\subset \rm{Conv}_p(L)$$ as claimed\footnote{The last inclusion follows from the fact that in a finite dimensional space,  if $X_l \longrightarrow X_\infty$, $z_l$ is in the convex hull of $X_l$,  $z_l \longrightarrow z$, then $z$ is in the convex hull of $X_\infty$ as it is easily proved using Caratheodory's theorem.}. 
\end{proof} 

\section{Appendix 2: The set of differentiability points for a selector.}\label{Appendix-11}
This Appendix is reproduced and translated  from \cite{Ottolenghi-Viterbo}. 
Let $L$ be a smooth  Lagrangian submanifold in $T^*N$ having $S(x,\xi)$ as  Generating Function Quadratic at Infinity. 
We recall that  $\Sigma_S=\{(x,\xi) \mid \frac{\partial S}{\partial \xi}(x,\xi)=0 \}$ and $i_S: \Sigma_S \longrightarrow T^*N$ given by $i_S(x,\xi)=(x, \frac{\partial S}{\partial x}(x,\xi))$. 

Let $u_L(x)=u_S(x)=c(1_x, S_x)$ be the corresponding selector. 
\begin{proposition} (\cite{Ottolenghi-Viterbo}, stated in \cite{Chaperon} without proof)
There exists a closed set of zero measure, $Z_L$, such that $u_L$ is smooth on $N\setminus Z_L$ and on this set $(x,du_L(x))\in L$. 
\end{proposition} 
\begin{proof} 
Let $Z^1_L$ be the set of singular values of the projection $\pi : L \longrightarrow N$. Then y Sard's theorem,  $Z_L^1$ is closed of zero measure. 
Let $U$ be  a connected component of $N\setminus Z^1_L$. Then $\pi$ restricted to $\pi^{-1}(U)$ is a covering. We set

\begin{gather*} Z_L^2=\left\{ x\in N \mid \exists \eta\neq \eta', \; \frac{\partial S}{\partial \xi}(x,\eta) =  \frac{\partial S}{\partial \xi}(x,\eta')=0, \; S(x,\eta)=S(x,\eta'), \right . \\ \left .  \frac{\partial S}{\partial x}(x,\eta) \neq   \frac{\partial S}{\partial x}(x,\eta') \right \}
\end{gather*} 
We claim that $Z_L^2$ is closed in $N\setminus Z^1_L$. Let us argue by contradiction and assume we had a sequence $(x_n)_{n\geq 1}$ in $ Z_L^2$ having limit $x$. We must prove that either $x\in Z_L^2$ or $x\in Z_L^1$.  Let $\eta_n, \eta'_n$ be a sequence corresponding to $x_n$, so that $\eta_n\neq \eta'_n$. By extracting a subsequence, we may assume $\eta_n \longrightarrow \eta, \eta'_n \longrightarrow \eta'$. 

Then, 
\begin{itemize} 
\item either $\eta\neq \eta'$ and then we cannot have  $\frac{\partial S}{\partial x}(x,\eta)=  \frac{\partial S}{\partial x}(x,\eta')$ otherwise $i_S$ would not be an embedding. Thus  we have  $x \in Z_L^2$. 
\item or we have $\eta=\eta'$. Then if $u_n=i_S(x_n,\eta_n), u'_n=i_S(x_n, \eta'_n)$ we have $u_n\neq u'_n$ and $\pi(u_n)=\pi(u'_n)=x_n$. Setting $z=\lim_nu_n=\lim_n u'_n$ we have $z\in L$ since $L$ is closed, and $\pi(z)=x$. Since $x\notin Z_L^1$,  $d\pi(z)$ is onto, so $\pi$  yields a local diffeomorphism between a neighborhood $W$ of $z$ in $L$ and a neighborhood of $x$ in $N$. Then for $n$ large enough, $u_n, u'_n$ will be in $W$ but this would imply $u_n=u'_n$ a contradiction. As a result $x \in Z_L^1$. \end{itemize} 
Assume now that $Z_L^2 \cap (N\setminus Z_L^1)$ is not a set of zero measure and 
Let now $x_0\in Z_L^2 \cap (N\setminus Z_L^1)$. 
Since $\pi$ is a covering near $x_0$ consider  a  a neighborhood $U$ of $x_0$. Given $(x_0,\eta), (x_0,\eta')$ in $\pi^{-1}(x_0)$  we may find smooth functions $\eta(x), \eta'(x)$ defined on $U$ and coinciding with $\eta, \eta'$ at $x_0$ such that $(x,\eta(x)), (x,\eta'(x))\in \Sigma_S$. 

Then the set of $x$ such that $$S(x,\eta(x))=S(x,\eta'(x)) \;\text{and}\;  \frac{\partial S}{\partial x}(x,\eta (x)) \neq   \frac{\partial S}{\partial x}(x,\eta'(x))$$ has zero measure since this is equal, setting $f(x)=S(x,\eta(x))-S(x,\eta'(x))$ to the    set $f(x)=0, df(x)\neq 0$ :  this is a nonsingular hypersurface, so must have measure   zero. Since we can only have a countable number of sheets over $x_0$, we get that $Z_L^2 \cap (N\setminus Z_L^1)$ is a countable union of sets of measure zero, so has measure zero. Since $Z_L^1$ has also measure zero, this proves our first claim. 

We now set $Z_L=Z_L^1\cup Z_L^2$.  
Our proof will be concluded if we can find a smooth map $\eta$ defined on $N\setminus Z_L$ such that 
 $u_S(x)=S(x,\eta(x))$. But if $x_0\notin Z_L$, we can find a neighborhood $U$ of $x_0$ such that $\pi $ is a trivial covering from  $\pi^{-1}(U)$ to $U$. Consider the various smooth sections, $\eta_j(x)$. Since $x\notin Z_L^2$ we cannot have $S(x,\eta_j(x))=S(x, \eta_k(x))$ for $j\neq k$ unless $ \frac{\partial S}{\partial x}(x,\eta_j(x)) =   \frac{\partial S}{\partial x}(x,\eta_k(x))$, but then $L$ would not be embedded. 
 
  So $u_S$ defines a unique $j$ such that $u_S(x)=S(x,\eta_j(x))$ in $U$ and then $u_S$ is smooth in $U$. Since $\eta_j$ is smooth, we have $$ du_S(x)= \frac{\partial S}{\partial x}(x,\eta_j(x))+  \frac{\partial S}{\partial \xi}(x,\eta_j(x))$$
but since the last term is zero, we get $$(x,du_S(x))=(x,du_L(x))=\left (x, \frac{\partial S}{\partial x}(x,\eta_j(x))\right )\in L$$
This concludes our proof. 
\end{proof} 
\begin{remark} 
Let us mention here a result of Seyfaddini and the author, that is mentioned in \cite{Vichery2}.
Let $L_k$ be a sequence $\gamma$-converging to a smooth Lagrangian $L$. Then $L\subset \lim_k L_k$, that is for each $z \in L$ there is a sequence $z_k\in L_k$ such that $\lim_k z_k=z$. This is a direct consequence  of lemma 7 in \cite{H-L-S}. This can be proved directly as follows. Indeed, if this was not the case, we would have $B(z, r)$ such that $B(z,r)\cap L_k=\emptyset$. Then for any $\varphi_H$ generated by a Hamiltonian $H$  supported in $B(z,r)$, we have $\gamma (L_k,\varphi (L_k))=0$, hence $\gamma (L, \varphi(L))=0$. But it is easy to see by a local construction that this does not hold for all $\varphi$ supported near $z$. 
\end{remark} 
\section{The structure of \texorpdfstring{$\meas_{\alpha}$}{mualpha}}
It would be interesting to understand the structure of $\meas_{\alpha}$. In the convex case, the support of $\meas_{\alpha}$ is contained in the  graph of a Lipschitz $1$-form (i.e. $p_\alpha +du_\alpha$ where $u\in C^{1,1}$), hence is Lagrangian in a generalized sense. 
Here, the support of $\meas_{\alpha}$ cannot be a graph, since replacing $H$ by $H\circ \psi$ replaces $\meas_{\alpha}$ by $\psi_{*}(\meas_{\alpha})$, hence $\supp (\meas_{\alpha})$ is replaced by $\psi (\supp (\meas_{\alpha}))$. 

\begin{question}
Can one replace the support of $\meas_{\alpha}$ by an invariant Lagrangian current, that is a current $T_{\alpha}$ such that $T_{\alpha}\wedge \omega=0$, $\dim (\supp (T_{\alpha}))=n$, and $(\phi^t)_{*}(T_{\alpha})=T_{\alpha}$  ? 
\end{question}

A question we did not answer until now is the location of the support of the measure compared  to the support of $H$.

\begin{proposition} 
For $\alpha \neq 0$, the support of $\meas$ is partially contained in the interior of $\supp (H)$, that is $\meas_{\alpha}(interior ({\supp (H)}))>0$.
\end{proposition} 
\begin{proof} 

Indeed if a measure is contained in the complement of the interior of $\supp(H)$, we have that its rotation vector is given by 
$$\tau \mapsto \int_{T^*T^n} \tau(q) \frac{\partial H}{\partial p}(q,p) d\meas$$ 
and this vanishes. So if the measure $\meas_\alpha$ was supported in the set where $\nabla H=0$, we would have $\rho(\meas_\alpha)=0$, contrary to the assumption.
  
\end{proof} 
It is also not difficult to say more in the case that $H$ is time-independent. Since the orbit of a point remains in a fixed energy level, and the same will be true
for the limit of the measure supported on such orbits. As a result we get the following result, proved for minimal measures in the Lagrangian situation in \cite{D-C}.

\begin{proposition} Assume $H$ is autonomous. Then, for any $\alpha\in d_s\overline H (p) $ the measure $\meas_{\alpha, p}$ constructed in Theorem \ref{main-theorem} is supported on a level set $\{(x,p) \mid H(x,p)= c\}$.   Moreover  we have  ${\mathcal A}(\meas_{\alpha})=p\cdot \alpha -c$.
\end{proposition} 
\begin{proof} 
Indeed, each of the trajectories $\gamma_k$ is contained in some $H^{-1}(c_k)$. If we select a subsequence such that $c_k$ converges to some value $c$, then we have that $\meas_\alpha$ is supported in $H^{-1}(c)$. 
\end{proof} 
This implies that for $\alpha \neq 0$, the measure is supported at  a positive distance from the complement of the support.

\begin{question}
 Is this still true for the time dependent case ? 
\end{question}

\end{document}